\documentclass[12pt]{amsart}
\usepackage[utf8]{inputenc}
\usepackage[margin=1in]{geometry} 
\usepackage{amsmath,amsthm,amssymb,amsfonts}
\usepackage{enumitem}
\usepackage{mathtools}

\newcommand{\Z}{\mathbb{Z}}
\usepackage{xcolor}

\usepackage{tikz}
\usetikzlibrary{arrows}

\usepackage{thmtools}
\usepackage{bbm}

\newtheorem{theorem}{Theorem}[section]
\newtheorem{lemma}[theorem]{Lemma}
\newtheorem{proposition}[theorem]{Proposition}
\newtheorem{corollary}[theorem]{Corollary}
\newtheorem{definition}[theorem]{Definition}
\theoremstyle{definition}
\newtheorem{example}[theorem]{Example}
\newtheorem{remark}[theorem]{Remark}

\newcommand{\be}{\begin{equation}}
\newcommand{\ee}{\end{equation}}
\newcommand{\de}{\begin{align*}}
\newcommand{\fe}{\end{align*}}

\newcommand{\Zline}{\mathbb Z\times\mathbb Z_{\geq0}}

\usepackage{bm}
\usepackage{graphicx}
\usepackage{subfig}
\usepackage{float}
\usepackage{hyperref}

\title[The censored stochastic six-vertex model]{The censored stochastic six-vertex model and parabolic Kazhdan--Lusztig $R$-polynomials}

\author[H.\ Drillick]{Hindy Drillick}
\address{H.\ Drillick,
	Courant Institute of Mathematical Sciences,
	\newline\hphantom{\quad \ \ H. Drillick}
	 251 Mercer St, New York, NY 10012 USA
}
\email{hindy.drillick@nyu.edu}

\author[L.\ Haunschmid-Sibitz]{Levi Haunschmid-Sibitz}
\address{L.\ Haunschmid-Sibitz,
	KTH Royal Institute of Technology,
	\newline\hphantom{\quad \ L. Haunschmid-Sibitz}
	Lindstedtsvägen 25, SE - 100 44 Stockholm, Sweden
}
\email{levihs@kth.se}

\keywords{
	Stochastic six-vertex model, censoring inequality, parabolic Kazhdan--Lusztig $R$-polynomials.
}

\begin{document}

\begin{abstract}
    We introduce a censored version of the stochastic six-vertex model. We show that for parameters $b_1 < b_2$, this model started from the initial condition $\bm{1}_{x>0}$ is stochastically dominated at any time by the blocking measure. 
    This is a partial analog of the censoring inequality for monotone spin systems.
    In particular, this result allows us to control the behavior of second-class particles.
    The proof uses parabolic Kazhdan--Lusztig $R$-polynomials, whose appearance is explained using a connection between the stochastic six-vertex model and the Iwahori--Hecke algebras of symmetric groups.
    Furthermore, we find an intertwining relation for this process using normalized parabolic Kazhdan--Lusztig $R$-polynomials as an intertwining kernel.
\end{abstract}

\maketitle



\section{Introduction}
\subsection{Preface}
The stochastic six-vertex model was first introduced by \cite{MR1147356} as a specialization of the six-vertex model, which is a classical model from equilibrium statistical mechanics going back to \cite{pauling1935structure}.
With this specialization, the model has a Markovian structure, which enables it to be seen as a discrete-time interacting particle system. It is connected via a suitable limit degeneration to the asymmetric simple exclusion process (ASEP) \cite{MR3589552}, the Kardar--Parisi--Zhang (KPZ) equation \cite{MR3650415,MR4091508, MR4046044}, the stochastic telegraph equation \cite{MR4038051, MR3951443, MR4193889}, and lies in the (one-dimensional) KPZ universality class \cite{MR1147356, MR3466163, MR4561796, aggarwal2024kpz, aggarwal2024scaling}.
Furthermore, it can be put into the more general setting of higher-spin vertex models, see \cite{MR3477349, MR3782413, MR4518478}.

In this paper, our focus is on certain monotonicity properties of the stochastic six-vertex process, which are much more subtle than their counterparts for other exclusion processes such as ASEP.
The stochastic six-vertex process has a censored version, where at certain censored vertices, particles and holes cannot pass each other. Our main result, Theorem~\ref{thm:censoring}, shows that the censored stochastic six-vertex process is still dominated by its stationary measure at any fixed time.
This can be used to bound the behavior of higher-class particles, see Theorem~\ref{thm:labeling}, which was the original motivation for this work.

This result can also be seen as a partial \emph{censoring inequality}. Censoring inequalities originate in  \cite{MR3106501} for monotone spin systems. A censoring inequality states that censoring certain vertex updates can only increase the (total variation) distance to stationarity, and that the censored process will, at any fixed time, be dominated by the stationary measure.
We call our result a \emph{partial} censoring inequality since, in our setting, it is no longer true that censoring always increases the distance to stationarity.


In the proof of the main results, we introduce a family of measures which we call \emph{parabolic Kazhdan--Lusztig R-measures}, as they are given by normalized parabolic Kazhdan--Lusztig R-polynomials. The appearance of Kazhdan--Lusztig R-polynomials in the context of the stochastic six-vertex model can be understood by considering the connection between the colored stochastic six-vertex model and Iwahori--Hecke algebras as described in \cite{bufetov2020interacting, MR4317703, MR4518478}.
Our result gives a new avenue to interpret properties of the Kazhdan--Lusztig R-polynomials probabilistically.

\subsection{The model}

\begin{figure}[ht]
		\centering
		\begin{tabular}{|c|c|c|c|c|c|c|}
			\hline
			Type & I & II & III & IV & V & VI \\
			\hline
			\begin{tikzpicture}[scale = 1.5]
			\draw[fill][white] (0.5, 0) circle (0.05);
			\draw[thick][white] (0, 0) -- (1,0);
			\draw[thick][white] (0.5, -0.5) -- 
			(0.5,0.5);
			\node at (0.5, 0) {Configuration};
			\end{tikzpicture}
			&
			\begin{tikzpicture}[scale = 1.2]
			\draw[fill] (0.5, 0) circle (0.05);
			\draw[thick] (0, 0) -- (1,0);
			\draw[thick] (0.5, -0.5) -- (0.5,0.5);
			\end{tikzpicture}
			&
			\begin{tikzpicture}[scale = 1.2]
			\draw[thick][white] (0, 0) -- (1,0);
			\draw[thick][white] (0.5, -0.5) -- (0.5,0.5);
			\draw[fill] (0.5, 0) circle (0.05);
			\end{tikzpicture}
			&
			\begin{tikzpicture}[scale = 1.2]
			\draw[thick][white] (0, 0) -- (1,0);
			\draw[thick] (0.5, -0.5) -- (0.5,0.5);
			\draw[fill] (0.5, 0) circle (0.05);
			\end{tikzpicture}
			&
			\begin{tikzpicture}[scale = 1.2]
			\draw[thick][white] (0, 0) -- (0.5,0);
			\draw[thick][white] (0.5, 0) -- (0.5, 0.5);
			\draw[thick] (0.5, 0) -- (1, 0);
			\draw[thick] (0.5, -0.5) -- (0.5, 0);
			(0.5,0.5);
			\draw[fill] (0.5, 0) circle (0.05);
			\end{tikzpicture}
			&
			\begin{tikzpicture}[scale = 1.2]
			\draw[thick] (0, 0) -- (1,0);
			\draw[thick][white] (0.5, -0.5) -- (0.5,0.5);
			\draw[fill] (0.5, 0) circle (0.05);
			\end{tikzpicture}
			&
			\begin{tikzpicture}[scale = 1.2]
			\draw[thick] (0, 0) -- (0.5,0);
			\draw[thick] (0.5, 0) -- (0.5, 0.5);
			\draw[thick][white] (0.5, 0) -- (1, 0);
			\draw[thick][white] (0.5, -0.5) -- (0.5, 0);
			(0.5,0.5);
			\draw[fill] (0.5, 0) circle (0.05);
			\end{tikzpicture}
			\\
			\hline
			Weight 
			& 1 & 1 & $b_1$ & $1- b_1$ & $b_2$ & $1-b_2$\\
			\hline
		\end{tabular}
		\caption{The six allowed configurations for the stochastic six-vertex model.}
		\label{fig:s6v}
	\end{figure}
    
We define the stochastic six-vertex model as a measure on collections of up-right paths in the upper half plane $\Zline$. We interpret the first coordinate of vertices $(x,t) 
\in \Zline$ as the spatial coordinate and the second coordinate as time.
All paths start on the horizontal line $\{(x,0):x\in\mathbb Z\}$.
Paths are allowed to share vertices, but not edges. This leads to six possible configurations at each vertex, each receiving a weight according to the two model parameters $b_1,b_2\in(0,1)$, see Figure \ref{fig:s6v}.
We take $b_1,b_2$ such that $0<b_1<b_2<1$ and consider them fixed throughout the paper.
The measure will further depend on an initial condition $\eta_0:\mathbb Z\to\{0,1\}$ where $\eta_0(x)=1$ if there is a path starting at the vertex $(x,0)$ and $\eta_0(x)=0$ otherwise.

Given an initial configuration $\eta_0$ such that $\eta_0(x) = 0$ for all $x$ sufficiently far to the left, i.e., there is a leftmost path and parameters $b_1<b_2$, the stochastic six-vertex measure can be defined by going line by line starting at $t=0$ and moving upwards and then left to right along each line, choosing at each vertex one of the possible configurations with probability given by its weight.
We define the occupation variables $(\eta_t(x))_{(x,t)\in\Zline}$ where $\eta_t(x) = 1$ if there is an incoming path on the edge below $(x+t,t)$ and $\eta_t(x)= 0$ otherwise.

\begin{definition}\label{def:singleclass}
    The process $(\eta_t)_{t\geq0}$ is called the (shifted) stochastic six-vertex process on the line with initial condition $\eta_0$ and parameters $b_1,b_2$.
\end{definition}

\begin{remark}
    The reason for including the shift $x+t$ in the definition of the stochastic six-vertex process is that the natural direction of time for the process is $(1,1)$.
    One way to see this is that only with this shift will the process possess non-translation-invariant stationary measures, while without it only translation-invariant measures will remain, see \cite{MR4660692}.
    Note also that in the scaling limit $b_1,b_2\to0$ with $\frac{b_1}{b_2}=q$ fixed it is the shifted process that converges to ASEP, see \cite{MR3589552}.
\end{remark}

This process has the same stationary measures as the asymmetric simple exclusion process, see \cite{MR4660692}.
In particular, these stationary measures are the blocking measures first defined in \cite{MR418291}, which depend on $b_1$ and $b_2$ through $q\coloneqq\frac{b_1}{b_2}\in(0,1)$.

\begin{definition}[Blocking Measure]

Let 
\begin{equation} \label{eq:ADef}
    A = \{\eta\colon\Z\to\{0,1\}: \sum_{x = -\infty}^0 \eta(x) = \sum_{x = 1}^{\infty}(1- \eta(x))<\infty\}
\end{equation} denote the set of all possible configurations that can be obtained if we start with the initial condition $\eta(x) = 1_{x > 0}$ and perform a finite number of moves. Further let $$\hat{\pi} = \otimes_{x \in \Z} \text{Ber}\left(\frac{q^{-x}}{1 + q^{-x}}\right).$$
We then define the blocking measure $\pi$ to be the projection of the measure $\hat{\pi}$ to the set $A$, i.e., $\pi(\cdot)=\hat{\pi}(\cdot|A)$.\footnote{We have an explicit formula for $  \hat \pi(A)$ (see e.g. \cite{MR3765898}):
$\hat \pi(A) = \frac{1}{ \sum_{l = -\infty}^{\infty} q^{l(l+1)/2}}.
$}
\end{definition}
\begin{proposition}[Proposition 6.4 in \cite{MR4660692}]
    $\pi$ is a stationary measure for the shifted stochastic six-vertex process.
\end{proposition}

We will use the following partial ordering of configurations in $A$ which is stated in terms of height functions. 
\begin{definition}[Height functions]\label{def:heightfun}
    For any $\eta\in A$ define its height function $h_\eta\colon\Z\to\Z$ by setting $\lim_{x\to-\infty}h_\eta(x)=0$, and $h_\eta(x+1)-h_\eta(x)=\eta(x)$.
\end{definition}
For two distributions $\mu$ and $\nu$ on $A$, we say that $\mu$ \emph{stochastically dominates} $\nu$, and write $\nu \preceq\mu$ if there is a coupling of $\nu$ and $\mu$, such that for $(\eta,\xi)$ sampled according to this coupling, $h_\xi(x)\leq h_\eta(x)$ for all $x$.

\begin{remark}
    Note that  $h_\xi \leq h_\eta $ holds if and only if the configuration $\xi$ can be obtained from $\eta$ by only moving particles to the right.
    The unique minimal configuration with respect to this partial order is given by $\bm{1}_{x>0}$.
\end{remark}

\subsection{Main theorem}\label{sec:intromain}
Consider the dynamics of the stochastic six-vertex model modified as follows.
\begin{definition}[Censored stochastic six-vertex model]\label{thm:s6vcensoring}
    Given a set $\mathcal C\subset\Zline$ of vertices, which we call censored vertices, we define the $\mathcal C$-censored stochastic six-vertex process $\eta_t$ as follows. The dynamics of $\eta_t$ are defined as in Definition \ref{def:singleclass}, except that whenever one updates at a vertex in $\mathcal C$, configurations III and V in Figure~\ref{fig:s6v} are never chosen, i.e., at the vertices in $\mathcal C$, paths can never go straight, they must always turn.
    See Figure \ref{fig:censoredLattice} for a visual representation of the censored vertices.
\end{definition}

Our main theorem is the following:
\begin{theorem}\label{thm:censoring}
    Let $0 < b_1 < b_2 < 1$ and let $\mathcal C$ be any subset of $\Zline$. Then the  $\mathcal C$-censored stochastic six-vertex model $\eta$ started from $\eta_0=\bm{1}_{x>0}$ satisfies that the law of $\eta_t$ is dominated by the  blocking measure $\pi$ for any time $t$. 
\end{theorem}

\begin{figure}
    \centering
    \includegraphics[width=0.9\linewidth]{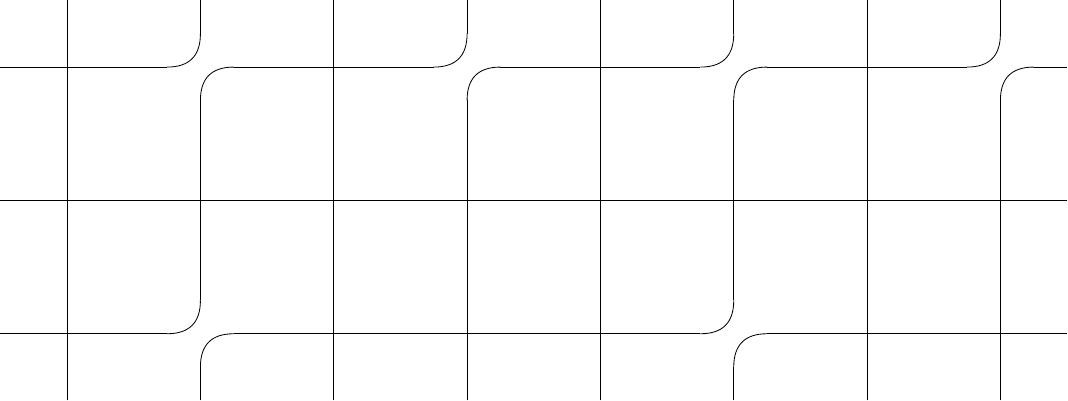}
    \caption{A piece of $\mathbb Z\times\mathbb Z_{\geq0}$ where some of the vertices have been censored.}
    \label{fig:censoredLattice}
\end{figure}

\begin{remark}
    The uncensored case of Theorem \ref{thm:censoring} (i.e., where $\mathcal C = \emptyset$) is also interesting and nontrivial. This case can also be proven using the monotone coupling in \cite[Proposition 2.6]{MR4093866}.
\end{remark}

Let us compare Theorem \ref{thm:censoring} to the censoring inequality from \cite{MR3106501}.
Informally, the censoring inequality can be stated as follows.
\begin{theorem}[Theorem 1.1 in \cite{MR3106501}]\label{thm:PWcensoring}
    Consider Glauber dynamics on a monotone spin system. 
    Let $\nu$ and $\tilde{\nu}$ both be obtained by starting from the minimal configuration of the spin system and running a censored Glauber dynamics where the set of transitions censored for $\tilde{\nu}$ contains the set censored for $\nu$.
Then we have that
\begin{equation}\label{eq:censoringPW}
\tilde{\nu}\preceq\nu\preceq\pi\,,
\end{equation}
where $\pi$ is the stationary measure of the Glauber chain.
\end{theorem}
It is not surprising that the full version of this theorem does not hold for the stochastic six-vertex model, since the stochastic six-vertex model lacks a form of monotonicity which is a key ingredient in the proof of Theorem~\ref{thm:censoring}.
Indeed one can easily see that the reverse inequality can hold between $\nu$ and $\tilde{\nu}$, see Example \ref{ex:censoring}.
Considering this, it is perhaps surprising that the second part of the inequality in \eqref{eq:censoringPW}, namely domination by the stationary measure, does still hold.

The proof of Theorem \ref{thm:censoring} proceeds by showing that at any time $t$ the law of $\eta_t$ is given by a convex combination of $v_\lambda$, where $(v_\lambda)_{\lambda\in\mathcal P}$ is a family of probability measures indexed by partitions that we define in Definition~\ref{def:vn}. The measures $v_{\lambda}$ are normalized versions of the \emph{parabolic Kazhdan--Lusztig $R$-polynomials}, a family of polynomials appearing in the study of representations of Hecke algebras. They were introduced in \cite{MR916182} as a projection of the (non-parabolic) Kazhdan--Lusztig $R$-polynomials \cite{MR560412} onto a parabolic quotient.
Theorem~\ref{thm:censoring} follows from the key observation that all the measures $v_\lambda$ are stochastically dominated by $\pi$. 
The measures $v_{\lambda}$ arose naturally in the proof, and their intriguing connection to parabolic Kazhdan--Lusztig $R$-polynomials will be illuminated in Section~\ref{sec:Hecke}.

When expanding $\eta_t$ as a convex combination of the measures $v_{\lambda}$, the coefficients in this expansion are themselves given by a stochastic six-vertex process with the parameters $b_1$ and $b_2$ flipped. This gives rise to an intertwining relation for the stochastic six-vertex model, which we state as Theorem \ref{thm:sym}.

For ASEP the corresponding statement to Theorem \ref{thm:censoring} follows easily from the ideas used in \cite{MR3106501}, see \cite{MR3474475} for SEP and \cite{ferrari2025mixingtimesopenasep} for ASEP (see also \cite{MR4564424}).
For completeness we write out the steps of this proof in Appendix~\ref{sec:ASEPregime}.
For the stochastic six-vertex process this is only possible if $b_1+b_2\leq1$, as otherwise the process lacks the kinds of monotonicity properties necessary for the proof, i.e., the ones in Lemmas~\ref{lem:decreasing} and \ref{lem:monotonicity}.
One should note that even in the case of $b_1+b_2\leq 1$ where Theorem~\ref{thm:censoringASEP} yields a stronger result than Theorem~\ref{thm:censoring}, the proof of Theorem~\ref{thm:censoring} still gives new insight into the structure of the process.

One useful application of Theorem \ref{thm:censoring} is to control the behavior of second-class particles in the multi-class (or colored) stochastic six-vertex process. We can do this by applying Theorem \ref{thm:censoring} to a censoring scheme where we censor vertices corresponding to first-class particles and holes such that the resulting graph describes the behavior of second-class particles with respect to third-class particles, i.e., where all positions containing first-class particles and holes have been deleted. This is stated precisely in Theorem \ref{thm:labeling}. A stochastic domination result of this type allows for greater control over these higher-class particles starting from non-stationary initial conditions. Such coupling techniques to control second-class particles are widely utilized in the study of interacting particle systems, see e.g. \cite{MR2135733, MR2630064, MR2485877, MR2919202, MR5002251, balazs2025localconvergencetpng, bernal2025limitprofilesasep}, but have not yet been fully developed for the stochastic six-vertex model.

\subsection{Organization of the paper}
In Section \ref{sec:proofMain}, we introduce the main technical theorem of the paper, Theorem~\ref{thm:product}, and show how it implies Theorem \ref{thm:censoring}. We prove Theorem~\ref{thm:product} in Section~\ref{sec:measures} by introducing the parabolic Kazhdan--Lusztig $R$-measures $v_{\lambda}$.
In Section~\ref{sec:app}, we show several consequences of the proof of the main theorem.
In Section~\ref{sec:Hecke}, we interpret our results in the language of Hecke algebras and Kazhdan--Lusztig $R$-polynomials. Finally, in Appendix \ref{sec:ASEPregime}, we explain how to prove the full censoring inequality if $b_1 + b_2 \leq 1$. 

\subsection{Acknowledgments}
We thank Dominik Schmid for suggesting this problem and for many helpful discussions. We thank Che Shen for explaining some properties of Kazhdan--Lusztig $R$-polynomials and in particular for suggesting a simple proof of Proposition \ref{prop:probmeasure}. We also thank Ivan Corwin, Jimmy He, and Roger Van Peski for helpful discussions. This material is partly based upon work supported by the National Science Foundation under Grant No. DMS-2424139, while the first author was in residence at the Simons Laufer Mathematical Sciences Institute in Berkeley, California, during the Fall 2025 semester. Part of this work also originated during the second author's stay at the Hausdorff Research Institute for Mathematics during the program ``Probabilistic Methods in Quantum Field Theory".

\section{Proof of the main theorem} \label{sec:proofMain}
In this section we prove the main theorem by reducing it to Theorem~\ref{thm:product} below, which will be proved in Section~\ref{sec:measures}.

\begin{definition}
        Given a particle configuration $\eta$, let  $\eta^k$ denote the configuration obtained from $\eta$ by swapping the values of $\eta(k)$ and $\eta(k+1)$.
\end{definition}

\begin{definition}\label{def:Pk}
    Let $P_k$ be the propagator of the following dynamics on particle configurations $\eta\colon\mathbb Z\to\{0,1\}$. 
    Exchange the values of $\eta(k)$ and $\eta(k+1)$ with probability $b_1$ if $(\eta(k),\eta(k+1))=(0,1)$ and with probability $b_2$ if $(\eta(k),\eta(k+1))=(1,0)$.
    In other words, if $\nu$ is a measure on particle configurations and $\eta$ is a particle configuration, then
    \[
    P_k\nu(\eta)=\begin{cases}
        b_2\nu(\eta^k)+(1-b_1)\nu(\eta), &\text{ if }(\eta(k),\eta(k+1))=(0,1)\\
        b_1\nu(\eta^k)+(1-b_2)\nu(\eta), &\text{ if }(\eta(k),\eta(k+1))=(1,0)\\
        \nu(\eta), &\text{ if }\eta(k)=\eta(k+1)\,.
    \end{cases}
    \]
\end{definition}

\begin{remark}
    Note that the $P_k$ act on Dirac measures in the following way:
    \begin{equation}\label{eq:PkonDirac}
        P_k\delta_{\eta}=\begin{cases}
            (1-b_1)\delta_{\eta}+b_1\delta_{\eta^k}, &\text{ if }(\eta(k),\eta(k+1))=(0,1)\\
            (1-b_2)\delta_{\eta}+b_2\delta_{\eta^k}&\text{ if }(\eta(k),\eta(k+1))=(1,0)\\
            \delta_{\eta}, &\text{ if }\eta(k)=\eta(k+1)\,,
        \end{cases}
    \end{equation}
    as can be easily checked from the definition.
\end{remark}

Applying these operators in the correct order to the initial condition gives exactly the law of the stochastic six-vertex process. 

\begin{proposition}\label{prop:inter}
The law of the $\mathcal C$-censored stochastic six-vertex process $\eta_t$ started from the initial condition $\bm{1}_{x>0}$ is given by 
\[
\lim_{N\to\infty}\prod_{s=0}^{t-1}\prod_{k=-N}^{t-1-s} {P}_{-k,t-1-s}\delta_{\bm{1}_{x>0}}
\]
where 
\[
{P}_{k,t}=\begin{cases}
    P_k\,\text{, if }(k+t,t)\notin\mathcal C\\
    \mathrm{Id}\,\text{, else.}
\end{cases}
\]
\end{proposition}
\begin{remark}
    Note that we are following the standard notation of applying the operators in the product from right to left, i.e., the first operator applied to $\delta_{\bm{1}_{x>0}}$ is $P_{0,0}$.
    Unfortunately, this necessitates the reversal in the indices in the formula above.
\end{remark}
\begin{proof}
    Consider the intermediate states $\eta_{t,k}$ defined by
    \[
    \eta_{t,k}(x)=\begin{cases}
        \eta_{t+1}(x)\text{, if }x<k\\
        0\text{, if }x=k\text{ and the edge between $(k-1+t,t)$ and $(k+t,t)$ is empty}\\
        1\text{, if }x=k\text{ and the edge between $(k-1+t,t)$ and $(k+t,t)$ is occupied}\\
        \eta_t(x)\text{, if } x>k\,.
    \end{cases}
    \]
        \begin{figure}
        \centering
        \includegraphics[width=0.7\linewidth]{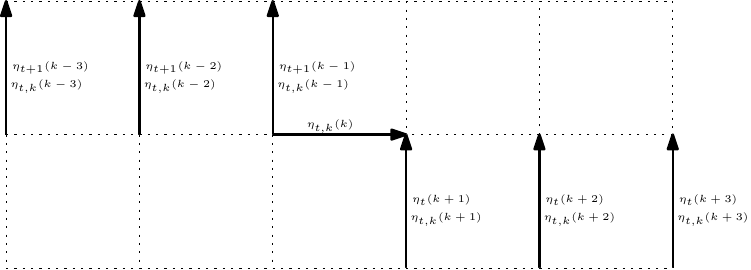}
        \caption{The edges whose occupation variables are the $\eta_{t,k}$ are shown as arrows.}
        \label{fig:intermediate}
    \end{figure}
    See Figure~\ref{fig:intermediate} for a diagram of this definition.
    Note the following properties of the $\eta_{t,k}$.
    First, note that for $k\leq-t$, we have $\eta_{t,k}=\eta_t$, since $\eta_{t+1}(x)=\eta_t(x)=0$ for $x\leq -t$ (because all particles initially started to the right of $0$) and therefore also $\eta_{t,k}(k)=0$.
    Furthermore $\lim_{k\to\infty}\eta_{t,k}=\eta_{t+1}$, with the sequence being almost surely constant for large enough $k$, because almost surely for $x$ large enough $\eta_{t+1}(x)=\eta_t(x)=1$.
    Note that this implies that the law of $\eta_{t,k}$ converges weakly to the law of $\eta_{t+1}$, even though this law will not be constant in $k$.
    Finally note that
    \begin{equation}\label{eq:singleflip}
        \mathcal L(\eta_{t,{k+1}})= P_{k,t}\mathcal L(\eta_{t,k})\,,
    \end{equation}
    where $\mathcal L$ denotes the law of a random variable.
    Indeed, $\eta_{t,k}(k)$ and $\eta_{t,k}(k+1)$ are the occupation variables of the incoming edges of the vertex $(k+t,t)$ and $\eta_{t,k+1}(k)$ and $\eta_{t,k+1}(k+1)$ are the occupation variables of the outgoing edges at this vertex.
    Therefore, if the vertex $(k+t,t)$ is not censored, $\eta_{t,k+1}$ is obtained from $\eta_{t,k}$ exactly by swapping the values $\eta_{t,k}(k)$ and $\eta_{t,k}(k+1)$ with the probabilities just as in Definition~\ref{def:Pk}, while if it is censored $\eta_{t,k+1}=\eta_{t,k}$.

    Iterating \eqref{eq:singleflip} together with the behavior of $\eta_{t,k}$ for $k\to\pm\infty$ discussed above gives
    \[
    \mathcal L(\eta_{t+1})=\lim_{N\to\infty}\prod_{k=-N}^t{P}_{-k,t}\mathcal L(\eta_t)\,.
    \]
    Iterating over $t$ yields
    \[
    \prod_{s=0}^{t-1}\lim_{N\to\infty}\prod_{k=-N}^{t-1-s} {P}_{-k,t-1-s}\delta_{\bm{1}_{x>0}}
    \]
     by noting that $\mathcal L(\eta_0)=\delta_{\bm{1}_{x>0}}$.
     Since $\eta_{t,k}(x)$ only depends on $(\eta_{s}(y))_{s\leq t+1,y\leq x}$ one can take the limit outside of the outer product.
\end{proof}

Theorem~\ref{thm:censoring} will now follow easily from the following theorem, which is the main technical result of the paper.

\begin{theorem}\label{thm:product}
    Let $\delta_{\bm{1}_{x>0}}$ be the Dirac delta on the configuration $\eta(x)=\bm{1}_{x>0}$.
    For any sequence $k_1,\dots,k_n$, the measure
    \[
    \prod_{i=1}^nP_{k_i}\delta_{\bm{1}_{x>0}}
    \]
    is stochastically dominated by the blocking measure $\pi$.
\end{theorem}

\begin{proof}[Proof of Theorem~\ref{thm:censoring}]
    By Proposition~\ref{prop:inter}, the law of $\eta_t$ is the limit of measures that are each dominated by $\pi$ by Theorem~\ref{thm:product}.
    This property is preserved under weak convergence, and thus the law of $\eta_t$ is also dominated by $\pi$.
\end{proof}

Note that Proposition~\ref{prop:inter} actually shows the equivalence of Theorem~\ref{thm:censoring} and Theorem~\ref{thm:product}, since any product of $P_{k_i}$ can be achieved by censoring sufficiently many vertices.

\section{Parabolic Kazhdan--Lusztig \texorpdfstring{$R$}{R}-measures}\label{sec:measures}
\subsection{Partitions}
Throughout the paper, it will be convenient to identify particle configurations with partitions. We denote by $\mathcal P$ the set of all partitions, i.e.,
    \[
    \mathcal P=\{\lambda = (\lambda_1,\lambda_2,\dots,\lambda_k):\lambda_1\geq\lambda_2\geq\dots\geq\lambda_k\geq 0\}\,,
    \]
    where we identify partitions of different lengths if one can be obtained from the other by removing trailing zeros.
   We define the size of a partition to be $|\lambda|=\sum_{i\geq1}\lambda_i$.

We further identify a partition $\lambda=(\lambda_1,\dots)$ with its Young diagram, i.e., the set
\[
\{(i,j)\in \mathbb Z_{\geq 1}^2: 1\leq j\leq \lambda_i\}\,.
\]
We draw Young diagrams as unions of $1\times1$ squares using the English convention, i.e., we index the rows $i$ from top to bottom and columns $j$ from left to right.

We identify particle configurations in $A$ with partitions by letting $X_1,X_2,\dots$ denote the ordered positions of the particles and then setting $\lambda_i=i-X_i$. In other words, $\lambda_i$ records the distance $X_i$ has traveled from its starting position. Note that increasing $\lambda_i$ by one is equivalent to $X_i$ jumping one step to the left. The configuration $\bm{1}_{x > 0}$ is identified with the empty partition, which we denote by~$()$.

The partial ordering of particle configurations by their height function corresponds to the partial ordering of partitions given by
\begin{equation}\label{eq:ordering}
    \lambda\leq\mu:\!\iff\lambda_i\leq\mu_i\text{, for all } i\geq 1\,.
\end{equation}
In terms of Young diagrams, this means that $\lambda \subseteq \mu$, so we can call this the \emph{containment ordering}.

Finally, we will also denote the set of partitions contained in a box of height $i$ and width $n-i$ by  $$\mathcal{P}(n, i) :=\left\{\lambda \in \mathcal{P}: \lambda \subseteq(n-i)^i\right\}.$$

From now on, we will abuse notation and identify measures on particle configurations with measures on partitions, e.g., we will write things like $\pi(\lambda)$.

\subsection{Properties of the blocking measure}

\begin{proposition}[q-Exchangeability] \label{prop:qExchange}
    The measure $\pi$ is $q$-exchangeable, meaning that if $\tilde{\eta}$ is obtained from $\eta$ by moving a particle one step to the left, then $\pi(\tilde{\eta})=q\pi(\eta)$.
    \end{proposition}
\begin{proof}
    This follows from the $q$-exchangeability of the product Bernoulli measure $\hat \pi$. 
\end{proof}

This allows us to compute the probability of a partition $\lambda$ under $\pi$.

\begin{proposition}\label{prop:measure_step} We have that 
    \begin{equation}
       \pi (\lambda) = q^{|\lambda|}\prod_{i=1}^{\infty} (1- q^i).
    \end{equation}
\end{proposition}

\begin{proof}
By $q$-exchangeability, the proof follows from showing that $\pi(())=\prod_{i=1}^{\infty} (1- q^i)$.
We have that 
    \begin{align*}
        1 &= \sum_{\lambda \in \mathcal P }\pi(\lambda ) \\
        &= \pi(()) \sum_{\lambda \in \mathcal P} q^{|\lambda|} \\
        &= \pi(()) \prod_{i =1}^{\infty} (1-q^i)^{-1}.
    \end{align*}
    The second line follows from $q$-exchangeability and the fact that to obtain the configuration $\lambda$ starting from the step configuration requires moving the $i$th leftmost particle $\lambda_i$ steps to the left.
\end{proof}

\subsection{Definition of parabolic Kazhdan--Lusztig \texorpdfstring{$R$}{R}-measures}

\begin{definition} We define the \emph{$k$th diagonal} of a partition to consist of all cells $(i, j)$ such that $i- j = k$, where $i$ denotes the row number and $j$ denotes the column number. 
\end{definition}

\begin{definition}
For $\lambda\in\mathcal P$ and $k\in\mathbb Z$, we define $\lambda^k$ as the partition obtained by flipping a corner in the $k$-th diagonal of $\lambda$, i.e., by adding or removing  a box with coordinates $(i,j)$ that satisfies $i-j=k$. If this is not possible, then we just take $\lambda^k = \lambda$.    
\end{definition}

For example, $()^0=(1)$, $(2,2,1,1)^0=(2,1,1,1)$, $(2,1)^{-2}=(3,1)$ and $(2,2)^1=(2,2)$.
    If $\lambda^k>\lambda$, then the corresponding particle configuration $\eta$ satisfies $(\eta(k),\eta(k+1))=(0,1)$, while if $\lambda^k<\lambda$ it satisfies $(\eta(k),\eta(k+1))=(1,0)$.

\begin{definition}[Rim Decomposition]
    A ribbon (also called a skew hook or border strip, see \cite{MR1354144}) of a partition $\lambda$ is a skew partition $\lambda/\mu$ that contains no $2\times2$ squares.
    The rim of a partition is its maximal ribbon.
    Successively removing the rims from a partition $\lambda$ gives a sequence of partitions $\lambda=\lambda^{(1)} \geq\lambda^{(2)} \geq\dots\geq ()$, such that $\lambda^{(k)}/\lambda^{(k+1)}$ is the rim of $\lambda^{(k)}$. We call this the rim decomposition of $\lambda$. See Figure \ref{fig:rim} for an illustration.    
\end{definition}

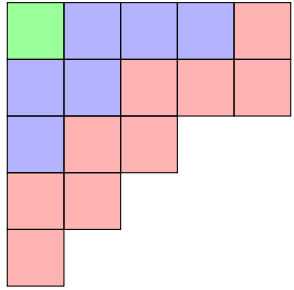
\begin{figure}[ht]
\centering
\begin{tikzpicture}[scale=.75, every node/.style={inner sep=0pt}]
  \fill[green!40] (0,0) rectangle (1,-1);
  \fill[blue!30] (1,0) rectangle (2,-1);
  \fill[blue!30] (2,0) rectangle (3,-1);
  \fill[blue!30] (3,0) rectangle (4,-1);
  \fill[red!30] (4,0) rectangle (5,-1);
  \fill[blue!30] (0,-1) rectangle (1,-2);
  \fill[blue!30] (1,-1) rectangle (2,-2);
  \fill[red!30] (2,-1) rectangle (3,-2);
  \fill[red!30] (3,-1) rectangle (4,-2);
  \fill[red!30] (4,-1) rectangle (5,-2);
  \fill[blue!30] (0,-2) rectangle (1,-3);
  \fill[red!30] (1,-2) rectangle (2,-3);
  \fill[red!30] (2,-2) rectangle (3,-3);
  \fill[red!30] (0,-3) rectangle (1,-4);
  \fill[red!30] (1,-3) rectangle (2,-4);
  \fill[red!30] (0,-4) rectangle (1,-5);

  \draw (0,0) rectangle (1,-1);
  \draw (1,0) rectangle (2,-1);
  \draw (2,0) rectangle (3,-1);
  \draw (3,0) rectangle (4,-1);
  \draw (4,0) rectangle (5,-1);
  \draw (0,-1) rectangle (1,-2);
  \draw (1,-1) rectangle (2,-2);
  \draw (2,-1) rectangle (3,-2);
  \draw (3,-1) rectangle (4,-2);
  \draw (4,-1) rectangle (5,-2);
  \draw (0,-2) rectangle (1,-3);
  \draw (1,-2) rectangle (2,-3);
  \draw (2,-2) rectangle (3,-3);
  \draw (0,-3) rectangle (1,-4);
  \draw (1,-3) rectangle (2,-4);
  \draw (0,-4) rectangle (1,-5);
\end{tikzpicture}
\caption{The rim decomposition of the partition $(5,5,3,2,1)$, where each color represents a rim.}
\label{fig:rim}
\end{figure}

\begin{definition}[Parabolic Kazhdan--Lusztig $R$-measures] \label{def:vn}
For a partition $\lambda$ with rim-decomposition $(\lambda^{(k)})_{k\geq 1}$ and a partition $\mu\leq\lambda$, define
\begin{equation}
    v_\lambda(\mu):= q^{|\mu|}\prod_{k\geq1}(1-q^k)^{n_k}\,,
\end{equation}
where $n_k$ is the number of connected components of the partition $\lambda^{(k)}$ (or equivalently of the rim $\lambda^{(k)}/\lambda^{(k+1)}$) when all boxes in $\mu$ are removed.
For connectivity, we consider only boxes sharing an edge, i.e., the components are `rook-connected'.
For $\mu\not\leq\lambda$ set $v_\lambda(\mu)=0$.\footnote{When we originally defined the measures $v_{\lambda}$, we were unaware of their connection to the parabolic Kazhdan--Lusztig $R$-polynomials, and it was only upon reading \cite{MR1972246} that we realized the precise connection. In \cite[Theorem 3.1]{MR1972246}, Brenti gives a combinatorial description of the parabolic Kazhdan--Lusztig $R$-polynomials, which turns out to be quite similar to the definition of the $v_{\lambda}$ we give here.}
\end{definition}
See Figure \ref{fig:vlambda} for an example of how to compute $v_{\lambda}$.
\begin{figure}[ht]
\centering
\begin{tikzpicture}[scale=0.75, every node/.style={inner sep=0pt}]
  \fill[gray!10] (0,0) rectangle (1,-1);
  \draw[gray!50, dashed] (0,0) rectangle (1,-1);
  \fill[gray!10] (1,0) rectangle (2,-1);
  \draw[gray!50, dashed] (1,0) rectangle (2,-1);
  \fill[gray!10] (2,0) rectangle (3,-1);
  \draw[gray!50, dashed] (2,0) rectangle (3,-1);
  \fill[gray!10] (0,-1) rectangle (1,-2);
  \draw[gray!50, dashed] (0,-1) rectangle (1,-2);
  \fill[gray!10] (1,-1) rectangle (2,-2);
  \draw[gray!50, dashed] (1,-1) rectangle (2,-2);
  \fill[gray!10] (2,-1) rectangle (3,-2);
  \draw[gray!50, dashed] (2,-1) rectangle (3,-2);

  \fill[blue!30] (3,0) rectangle (4,-1);
  \fill[red!30] (4,0) rectangle (5,-1);
  \fill[red!30] (3,-1) rectangle (4,-2);
  \fill[red!30] (4,-1) rectangle (5,-2);
  \fill[blue!30] (0,-2) rectangle (1,-3);
  \fill[red!30] (1,-2) rectangle (2,-3);
  \fill[red!30] (2,-2) rectangle (3,-3);
  \fill[red!30] (0,-3) rectangle (1,-4);
  \fill[red!30] (1,-3) rectangle (2,-4);
  \fill[red!30] (0,-4) rectangle (1,-5);

  \draw (3,0) rectangle (4,-1);
  \draw (4,0) rectangle (5,-1);
  \draw (3,-1) rectangle (4,-2);
  \draw (4,-1) rectangle (5,-2);
  \draw (0,-2) rectangle (1,-3);
  \draw (1,-2) rectangle (2,-3);
  \draw (2,-2) rectangle (3,-3);
  \draw (0,-3) rectangle (1,-4);
  \draw (1,-3) rectangle (2,-4);
  \draw (0,-4) rectangle (1,-5);
\end{tikzpicture}
\hspace{40pt}
\begin{tikzpicture}[scale=0.75, every node/.style={inner sep=0pt}]
  \fill[gray!10] (0,0) rectangle (1,-1);
  \draw[gray!50, dashed] (0,0) rectangle (1,-1);
  \fill[gray!10] (1,0) rectangle (2,-1);
  \draw[gray!50, dashed] (1,0) rectangle (2,-1);
  \fill[gray!10] (2,0) rectangle (3,-1);
  \draw[gray!50, dashed] (2,0) rectangle (3,-1);
  \fill[gray!10] (0,-1) rectangle (1,-2);
  \draw[gray!50, dashed] (0,-1) rectangle (1,-2);
  \fill[gray!10] (0,-2) rectangle (1,-3);
  \draw[gray!50, dashed] (0,-2) rectangle (1,-3);

  \fill[blue!30] (3,0) rectangle (4,-1);
  \fill[red!30] (4,0) rectangle (5,-1);
  \fill[blue!30] (1,-1) rectangle (2,-2);
  \fill[red!30] (2,-1) rectangle (3,-2);
  \fill[red!30] (3,-1) rectangle (4,-2);
  \fill[red!30] (4,-1) rectangle (5,-2);
  \fill[red!30] (1,-2) rectangle (2,-3);
  \fill[red!30] (2,-2) rectangle (3,-3);
  \fill[red!30] (0,-3) rectangle (1,-4);
  \fill[red!30] (1,-3) rectangle (2,-4);
  \fill[red!30] (0,-4) rectangle (1,-5);

  \draw (3,0) rectangle (4,-1);
  \draw (4,0) rectangle (5,-1);
  \draw (1,-1) rectangle (2,-2);
  \draw (2,-1) rectangle (3,-2);
  \draw (3,-1) rectangle (4,-2);
  \draw (4,-1) rectangle (5,-2);
  \draw (1,-2) rectangle (2,-3);
  \draw (2,-2) rectangle (3,-3);
  \draw (0,-3) rectangle (1,-4);
  \draw (1,-3) rectangle (2,-4);
  \draw (0,-4) rectangle (1,-5);
\end{tikzpicture}
\caption{An illustration of how to compute $v_{\lambda}(\mu)$ for  $\lambda = (5,5,3,2,1)$ and different choices of $\mu$. 
On the left we have $\mu = (3,3)$ and $v_{\lambda}(\mu)=q^6(1-q)^2(1-q^2)^2$, while on the right we have $\mu = (3,1,1)$ and $v_{\lambda}(\mu)=q^5(1-q)(1-q^2)^2$.}
\label{fig:vlambda}
\end{figure}
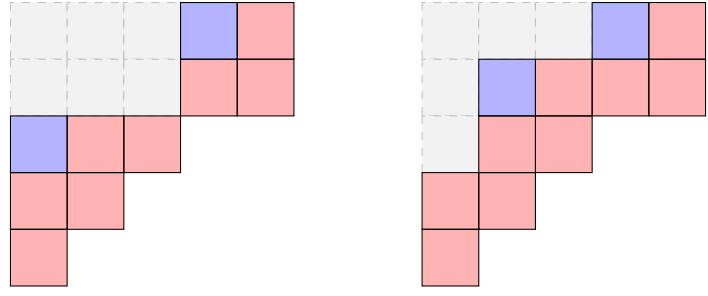

\begin{remark} We call these measures \emph{parabolic Kazhdan--Lusztig $R$ measures} because as we will show in Proposition \ref{prop:parabolicProjection}, we have 
    $$q^{-|\lambda|} v_{\lambda}(\mu) = R^{J, -1}_{u, v}\left(\frac1q\right)$$ where $R^{J, -1}_{u, v}\left(q \right)$ are certain parabolic Kazhdan--Lusztig $R$-polynomials defined in \cite{MR916182}. We define $u, v$ and $J$ in Proposition \ref{prop:parabolicProjection} in terms of parabolic quotients of the symmetric group.
\end{remark}

The functions $v_\lambda$, which we see as vectors in the vector space spanned by formal linear combinations of partitions, have the following alternative description which will be useful in the proofs below:

\begin{definition}[$q$-labeling of $\lambda$]
Given a partition $\lambda$, we will assign a number to each cell $(i,j)$ in $\lambda$. For each $k \in \Z$, we will label the cells in the $k$th diagonal according to the following rules:
\begin{enumerate}
    \item If $\lambda^k = \lambda$, we fill the $k$th diagonal with ones.
    \item If $\lambda^k > \lambda$, then we fill the $k$th diagonal with the values $(1-q)^{-1}, (1-q^2)^{-1},\dots$ starting with the box touching the bottom boundary.
    \item If $\lambda^k < \lambda$, we fill the $k$th diagonal with the values $(1-q),(1-q^2),\dots$, again starting with the box touching the bottom boundary.
\end{enumerate}
Let $w_\lambda(i,j)$ denote the value assigned to box $(i,j)$ by this labeling. See Figure \ref{fig:tiling} for an illustration of this labeling. 
\end{definition}

\begin{figure}[h]
\centering
\begin{tikzpicture}[scale=.75, every node/.style={inner sep=0pt}]
\tiny
  \draw (0,0) rectangle (1,-1);
  \draw (1,0) rectangle (2,-1);
  \draw (2,0) rectangle (3,-1);
  \draw (3,0) rectangle (4,-1);
  \draw (4,0) rectangle (5,-1);
  \draw (0,-1) rectangle (1,-2);
  \draw (1,-1) rectangle (2,-2);
  \draw (2,-1) rectangle (3,-2);
  \draw (3,-1) rectangle (4,-2);
  \draw (4,-1) rectangle (5,-2);
  \draw (0,-2) rectangle (1,-3);
  \draw (1,-2) rectangle (2,-3);
  \draw (2,-2) rectangle (3,-3);
  \draw (0,-3) rectangle (1,-4);
  \draw (1,-3) rectangle (2,-4);
  \draw (0,-4) rectangle (1,-5);

  \node at (0.5,-0.5) {$1{-}q^{3}$};
  \node at (1.5,-0.5) {$\frac{1}{1{-}q^{2}}$};
  \node at (2.5,-0.5) {$1$};
  \node at (3.5,-0.5) {$1{-}q^{2}$};
  \node at (4.5,-0.5) {$1{-}q$};
  \node at (0.5,-1.5) {$\frac{1}{1{-}q^{2}}$};
  \node at (1.5,-1.5) {$1{-}q^{2}$};
  \node at (2.5,-1.5) {$\frac{1}{1{-}q}$};
  \node at (3.5,-1.5) {$1$};
  \node at (4.5,-1.5) {$1{-}q$};
  \node at (0.5,-2.5) {$1{-}q^{2}$};
  \node at (1.5,-2.5) {$\frac{1}{1{-}q}$};
  \node at (2.5,-2.5) {$1{-}q$};
  \node at (0.5,-3.5) {$\frac{1}{1{-}q}$};
  \node at (1.5,-3.5) {$1{-}q$};
  \node at (0.5,-4.5) {$1{-}q$};
\end{tikzpicture}
\caption{The $q$-labeling of the partition $(5,5,3,2,1)$.}
\label{fig:tiling}
\end{figure}

\begin{lemma}[$q$-labeling description of $v_{\lambda}$]\label{lem:alternatevlambda}
    The function $v_\lambda$ has the following alternative description.
    \begin{equation}
        v_{\lambda}(\mu) = q^{|\mu|} \prod_{(i,j)\in\lambda/\mu}w_\lambda(i,j)\,.
    \end{equation}
\end{lemma}
\begin{proof}
    In the ribbon $\lambda^{(k)}/\lambda^{(k+1)}$, the boxes not filled with $1$ are filled with either $(1-q^k)$ or $(1-q^{k})^{-1}$.
    The values $(1-q^k)$ are placed in exactly the corners where the ribbon turns left (when traversing the ribbon from bottom left to top right) and the values $(1-q^k)^{-1}$ are in the boxes where the ribbon turns right.
    After removing the boxes from a partition $\mu$ one can see that every connected component will have exactly one more left turn than right turns, since they alternate, and the first and last turns must be left turns.
\end{proof}

\begin{definition}[Distance along $k$th diagonal]
    If $\mu \leq \lambda$ and if we have $\lambda^k > \lambda$ and $\mu^k > \mu$, then we define $$d_k(\mu, \lambda):= |\{(i,j) \in \lambda : i-j = k\}| - |\{(i,j) \in \mu : i-j=k\}|.$$
\end{definition}

The following proposition demonstrates that $P_k v_\lambda$ can always be written as a convex combination $v_{\lambda}$ and  $v_{\lambda^k}$.

\begin{proposition}[Convex Decomposition]\label{prop:vdynamics}
    For $k\in\mathbb Z$ we have
    \begin{equation}\label{eq:vdynamics}
        P_kv_\lambda=\begin{cases}
            (1-b_2)v_\lambda+b_2v_{\lambda^k}\text{, if $\lambda^k >\lambda$}\\
            (1-b_1)v_\lambda+b_1v_{\lambda^k}\text{, if $\lambda^k<\lambda$}\\
            v_\lambda\text{, else.}
        \end{cases}
    \end{equation}
\end{proposition}
\begin{proof}
    This statement will follow from the following properties of $v_\lambda$.
    \begin{enumerate}
        \item  If $\lambda^k=\lambda$ and $\mu^k>\mu$ we have
    \begin{equation}\label{eq:relation1}
    v_\lambda(\mu^k)=qv_\lambda(\mu)\,.
    \end{equation}
    
    \item If $\lambda^k>\lambda$ and $\mu^k=\mu$ we have
    \begin{equation}\label{eq:relationflatmu}
        v_{\lambda^k}(\mu)=v_\lambda(\mu)\,.
    \end{equation}
    \item If $\lambda^k>\lambda$ and $\mu^k>\mu$, we have
    \begin{equation}\label{eq:relationlambdamu}
        v_{\lambda^k}(\mu)=(1-q^{d_k(\mu, \lambda)+1})v_\lambda(\mu),\quad v_{\lambda}(\mu^k)=q(1-q^{d_k(\mu, \lambda)})v_\lambda(\mu)\text{, and }v_{\lambda^k}(\mu^k)=qv_\lambda(\mu)\,.
    \end{equation}
\end{enumerate}
    We have that \eqref{eq:relation1} and the second equation in \eqref{eq:relationlambdamu} follow from Lemma~\ref{lem:alternatevlambda} by examining which weight is assigned to the box being added to $\mu$.
    Proving the other equations just amounts to either checking how the product in Lemma \ref{lem:alternatevlambda} changes when going from $\lambda$ to $\lambda^k$, or by directly considering the product in Definition~\ref{def:vn}.

   By the definition of $P_k$ we have 
\begin{align}
     P_kv_{\lambda}(\mu)&=\begin{cases}\label{eq:P_kv_lbydef}
        b_2 v_{\lambda}(\mu^k)+(1-b_1)v_{\lambda}(\mu), &\text{ if }\mu^k > \mu \\
        b_1v_{\lambda}(\mu^k)+(1-b_2)v_\lambda(\mu), &\text{ if }\mu^k < \mu\\
        v_{\lambda}(\mu), &\text{ if }\mu^k  = \mu\,.
    \end{cases}
\end{align}

There are nine cases to check: for each of the three cases $\lambda^k>\lambda$, $\lambda^k=\lambda$ and $\lambda^k<\lambda$, we need to check that the right hand side of \eqref{eq:P_kv_lbydef} equals the right hand side of \eqref{eq:vdynamics} applied to $\mu$ again for $\mu$ such that $\mu^k>\mu$, $\mu^k=\mu$, or $\mu^k<\mu$.

The cases where $\mu^k=\mu$ follow from \eqref{eq:relationflatmu}, since \eqref{eq:relationflatmu} implies that the right hand side of \eqref{eq:vdynamics} equals $v_\lambda(\mu)$ for all $\mu$.
The cases where $\lambda^k=\lambda$ follow from \eqref{eq:relation1}:
\begin{align*}
    b_2v_\lambda(\mu^k)+(1-b_1)v_\lambda(\mu)&=b_2  q v_{\lambda}(\mu)+(1-b_1)v_{\lambda}(\mu)=v_\lambda(\mu)\\
    b_1v_{\lambda}(\mu^k)+(1-b_2)v_\lambda(\mu)&=b_1q^{-1}v_{\lambda}(\mu)+(1-b_2)v_\lambda(\mu)=v_\lambda(\mu)\,,
\end{align*}
where in both equations we used first \eqref{eq:relation1} (once for $\lambda$ and $\mu$ and once for $\lambda$ and $\mu^k$), and then $q=b_1/b_2$.

In the case where $\lambda^k<\lambda$ and $\mu^k>\mu$ we can use the third equation in \eqref{eq:relationlambdamu} to see that $v_{\lambda}(\mu^k)=qv_{\lambda^k}(\mu)$ to get
\[
b_2v_\lambda(\mu^k)+(1-b_1)v_\lambda(\mu)=qb_2v_{\lambda^k}(\mu)+(1-b_1)v_\lambda(\mu)=b_1v_{\lambda^k}(\mu)+(1-b_1)v_\lambda(\mu)\,.
\]
The case in which $\lambda^k>\lambda$ and $\mu^k<\mu$ is analogous using $v_{\lambda}(\mu^k)=q^{-1}v_{\lambda^k}(\mu)$, which again comes from the third equation in \eqref{eq:relationlambdamu}.

Let us now consider the case where $\lambda^k>\lambda$ and $\mu^k>\mu$.
Then the right hand side of \eqref{eq:P_kv_lbydef} equals
\begin{align*}
    b_2v_\lambda(\mu^k)+(1-b_1)v_\lambda(\mu)&=(b_2q(1-q^{d_k(\mu,\lambda)})+1-b_1)v_\lambda(\mu)\\
    &=(1-b_2q^{{d_k(\mu,\lambda)}+1})v_\lambda(\mu)\\
    &=((1-b_2)+b_2(1-q^{{d_k(\mu,\lambda)}+1}))v_\lambda(\mu)\\
    &=(1-b_2)v_\lambda(\mu)+b_2v_{\lambda^k}(\mu)\,,
\end{align*}
which equals the right hand side of \eqref{eq:vdynamics}. Here we used the second equation of \eqref{eq:relationlambdamu} in the first equality and the second equation of \eqref{eq:relationlambdamu} in the third equality.

Finally in the case $\lambda^k<\lambda$ and $\mu^k<\mu$ we can use first the first and third equations in \eqref{eq:relationlambdamu} to obtain $v_\lambda(\mu^k)=(1-q^{{d_k(\mu^k,\lambda^k)}+1})v_{\lambda^k}(\mu^k)=(1-q^{{d_k(\mu^k,\lambda^k)}+1})q^{-1}v_{\lambda}(\mu)$ and the second and third equations to obtain $v_{\lambda^k}(\mu)=(1-q^{d_k(\mu^k,\lambda^k)})v_\lambda(\mu)$. We then proceed as in the previous case.
\end{proof}

\begin{corollary}
    For any $\lambda$ we have
    \[
    \sum_\mu v_{\lambda}(\mu)=1\,.
    \] In other words, we can interpret $v_{\lambda}$ as a probability measure on the set of partitions. 
\end{corollary}

\begin{proof}
    If $\lambda = ()$ then this is true. We can then prove this inductively using  Proposition \ref{prop:vdynamics} to express $v_{\lambda^k}$ as a linear combination of $P_kv_\lambda$ and $v_\lambda$ with coefficients summing to $1$, which shows that the total mass of $v_{\lambda^k}$ is equal to $1$.
Non-negativity of $v_{\lambda^k}(\mu)$ is immediate from Definition~\ref{def:vn}.  
\end{proof}

\subsection{Proof of Theorem \ref{thm:product}}

\begin{definition}
    A decreasing subset $E$ of $\mathcal P$ is a set that satisfies
    \[
    \lambda\in E\text{ and } \mu\leq\lambda\implies\mu\in E\,.
    \]
\end{definition}
\begin{remark}
    Note that the conceptual advantage of working with decreasing instead of increasing events, is that they can be finite.
However, this is not actually important in what follows, and one could just as well have worked with increasing sets.
\end{remark}

\begin{proposition}\label{prop:monotonicity}
    For $\lambda\leq\lambda'$ it holds that $v_\lambda\preceq v_{\lambda'}$, i.e., for any decreasing set $E\subset\mathcal P$ and any $\lambda\leq\lambda'$ it holds that
    \begin{equation}
        v_\lambda(E)\geq v_{\lambda'}(E).
    \end{equation}
\end{proposition}
\begin{proof}
    Without loss of generality $\lambda'=\lambda^k$ for some $k$ since any $\lambda'\geq \lambda$ can be obtained from $\lambda$ by a sequence of adding boxes.
    Then we know that
    \[
    (P_kv_\lambda)(E)=(1-b_2)v_\lambda(E)+b_2v_{\lambda^k}(E)\,.
    \]
    Therefore, it suffices to show that $P_kv_\lambda(E)\leq v_\lambda(E)$.
    Let $\mu_0\sim v_\lambda$ and conditioned on $\mu_0$ let $\mathcal L(\mu_1|\mu_0)\sim P_k\delta_{\mu_0}$.
    Then $\mu_1\sim P_kv_\lambda$.
    Now we can write
    \begin{align}
        P_kv_\lambda(E)- v_\lambda(E)&=\mathbb P(\mu_1\in E)-\mathbb P(\mu_0\in E) \notag\\
        &=\mathbb P(\mu_1\in E,\mu_0\notin E)-\mathbb P(\mu_1\notin E,\mu_0\in E)\notag\\
        &=\sum_{\mu\in E:\mu^k\notin E}\left[\mathbb P(\mu_1=\mu,\mu_0=\mu^k)-\mathbb P(\mu_1=\mu^k,\mu_0=\mu) \right]\label{eq:mon}.
\end{align}
We can compute \eqref{eq:mon} in terms of $v_{\lambda}$ to obtain
\begin{align*}
        \sum_{\mu\in E:\mu^k\notin E}\left[b_2v_\lambda(\mu^k)-b_1v_\lambda(\mu)\right]&=\sum_{\mu\in E:\mu^k\notin E}\left[b_2q(1-q^{d_k( \mu, \lambda)})v_\lambda(\mu)-b_1v_\lambda(\mu)\right]\\
        &=-\sum_{\mu\in E:\mu^k\notin E}b_1q^{d_k( \mu, \lambda)}v_\lambda(\mu)\\
        &\leq 0.
    \end{align*}
\end{proof}

As $\lambda$ becomes bigger, the distribution $v_\lambda$ converges to the blocking measure.
\begin{lemma}
    If $(\lambda^{[n]})_{n\geq0}$ is any sequence of partitions such that $\lambda^{[n]}_i\to\infty$ for any $i$, then $v_{\lambda^{[n]}}(\mu)\to\pi(\mu)= q^{|\mu|}\prod_{i=1}^{\infty} (1- q^i)$.
\end{lemma}
\begin{proof}
    Consider first the square partition $b^{[n]}=(n,n,\dots,n)$.
   We have 
    \[
    v_{b^{[n]}}(\mu)=q^{|\mu|}\prod_{i=1}^{n-d(\mu)}(1-q^i)\,,
    \]
    where $d(\mu)$ is the side length of the Durfee square of $\mu$, i.e., the largest square that fits inside $\mu$.
    As $n\to\infty$, this clearly converges to  $q^{|\mu|}\prod_{i=1}^{\infty}(1-q^i)$ for any $\mu$.
    Now consider $\lambda^{[n]}$ as in the statement.
    Let $b(\lambda^{[n]})$ be the largest square partition contained in $\lambda^{[n]}$ and $B(\lambda^{[n]})$ the smallest square partition that contains $\lambda^{[n]}$.
    By the assumption on $\lambda^{[n]}$ both $b(\lambda^{[n]})$ and $B(\lambda^{[n]})$ increase to infinity as $n\to\infty$.
    By Proposition~\ref{prop:monotonicity} we have $v_{b(\lambda^{[n]})}\preceq v_{\lambda^{[n]}}\preceq v_{B(\lambda^{[n]})}$.
    Since both $v_{b(\lambda^{[n]})},v_{B(\lambda^{[n]})}$ converge to $\pi$, this implies that $v_{\lambda^{[n]}}(E)\to \pi(E)$ for any decreasing $E$, which implies that $v_{\lambda^{[n]}}\to\pi$.
\end{proof}

\begin{corollary}
    For any $\lambda$, $v_\lambda\preceq\pi$.
\end{corollary}
\begin{proof}
    This follows from $v_\lambda\preceq v_{b^{[n]}}\to\pi$, where the inequality holds for all $n$ large enough.
\end{proof}

We are now ready to prove Theorem~\ref{thm:product}.

\begin{proof}[Proof of Theorem~\ref{thm:product}]
Note that $\delta_{\bm{1}_{x>0}}=v_{()}$.
By applying Proposition~\ref{prop:vdynamics} repeatedly we have
    \begin{equation}\label{eq:convexcomb}
    \prod_{i=1}^nP_{k_i}\delta_{\bm{1}_{x>0}}=\sum_\lambda c_\lambda v_\lambda\,,
    \end{equation}
where $c_\lambda\geq0 $ for all $\lambda$ and $\sum_\lambda c_\lambda =1$.
Since all $v_\lambda$ are dominated by $\pi$, this implies that the product on the left is as well.
\end{proof}
\section{Further applications}\label{sec:app}
Now that Theorem~\ref{thm:censoring} has been established, we will consider what else its proof reveals about the stochastic six-vertex process.

\subsection{Symmetries}
A closer consideration of Proposition~\ref{prop:vdynamics} shows that the coefficients $c_\lambda$ appearing in the proof of Theorem~\ref{thm:product} are not just any convex combination. In particular, we get the following intertwining relation.
\begin{theorem}\label{thm:sym}
    Let $\mathbb P^{b_1,b_2}$ be the law of the $\mathcal C$-censored stochastic six-vertex process started from $\bm{1}_{x>0}$ as considered in Theorem~\ref{thm:censoring} and let $\mathbb P^{b_2,b_1}$ be the same process with the parameters $b_1$ and $b_2$ exchanged.
    Then for every time $t$ and every partition $\mu$
    \[
    \mathbb P^{b_1,b_2}(\eta_t=\mu)=\mathbb E^{b_2,b_1}(v_{\eta_t}(\mu))\,,
    \]
    where we identify particle configurations and partitions as usual.
\end{theorem}
\begin{proof}
    Notice that by Proposition~\ref{prop:vdynamics} the $P_k$ act on the measures $v_\lambda$ as it does on Dirac deltas $\delta_\lambda$, just with $b_1$ and $b_2$ exchanged.
    Therefore Proposition~\ref{prop:inter} gives that the law of $\eta_t$ is given by
    \[
    \sum_{\lambda\in\mathcal P}\mathbb P^{b_2,b_1}(\eta_t=\lambda)v_\lambda\,,
    \]
    which is equivalent to the statement written above.
\end{proof}



\begin{remark}
    Note that the same statement holds for ASEP, which can be seen by taking $b_1=\varepsilon L,b_2=\varepsilon R$ and $\varepsilon$ to $0$, or by noting that its law can be written as a random product of $P_k$'s.
    The switching of the parameters $b_1,b_2$ becomes the switching of left and right jump rates in that setting.
\end{remark}


\subsection{Half-line and segment}

The stochastic six-vertex model is also sometimes studied on the half line or on the strip, which are the domains $\{(x+t,t):x< 0\}$ and $\{(x+t,t):0\leq x \leq N\}$ for some $N>0$ respectively, see e.g. \cite{10.1215/00127094-2018-0019, 10.1214/24-EJP1100}.
Consider these processes with $k$ particles, starting all at the right edge of the domain.
Theorem~\ref{thm:censoring} applies to this situation, by shifting the domains such that the initial positions of the particles are given by $\{1,\dots,k\}$ and by censoring all vertices in $\{(k+1+t,t):t\geq0\}$ for the half-line and all vertices in $\{(k+1+t,t):t\geq0\}$ and $\{(-(N-k)-1+t,t):t\geq0\}$ for the strip.
However, we can obtain a stronger statement by considering which $P_k$ can appear in the product in Proposition~\ref{prop:inter} for these two domains.

\begin{definition}
    Configurations $\eta\colon\{x\in\mathbb Z:x\leq k\}\to\{0,1\}$ with exactly $k$ particles can be identified with partitions with at most $k$ rows, by extending them with value $1$ to the right of $k$ and applying the identification of particle configurations and partitions above.
    On these configurations define the measure
    \[
    \pi^k(\lambda)=v_{(\infty^k)}(\lambda)=q^{|\lambda|}\prod_{i=1}^k(1-q^i)\,,
    \]
    where $v_{(\infty^k)}$ is the limit $n\to\infty$ of $v_{(n,\dots,n)}$, i.e., the parabolic Kazhdan--Lusztig R-measure indexed by the partition with $k$ rows of length $n$. 
\end{definition}

\begin{proposition}[Half line]
    Consider the censored stochastic six-vertex process $(\eta_t)_{t\geq0}$ on the half line $\{(x+t,t):x\leq k\}$ with $k$ particles, started from the initial condition $\bm{1}_{0<x\leq k}$.
    Then, for any time $t$ the law of $\eta_t$ is dominated by $\pi^k$.
\end{proposition}
\begin{proof}
    Since all vertices on the boundary $\{(k+1+t,t):t\geq0\}$ are censored, $\eta_t(x)=1$ for all $x\geq k+1$. Therefore, it does not change the process if we censor all vertices to the right of the boundary as well, i.e., all vertices $\{(x,t):x\geq k+1+t\geq0\}$.
    Then in the product in Proposition~\ref{prop:inter}, no $P_j$ with $j>k$ will appear, which in turn will mean that the $c_\lambda$ appearing in the proof of Theorem~\ref{thm:product} will be $0$ for all partitions with more than $k$ rows and hence the law of $\eta_t$ is the limit of convex combinations of $v_\lambda$ with $\lambda$'s with at most $k$ rows.
    Each $v_\lambda$ with at most $k$ rows is dominated by $v_{(n,\dots,n)}$ for $n$ sufficiently large, therefore they are all dominated by $\pi^k$.
\end{proof}

\begin{proposition}[Strip]
        Consider the censored stochastic six-vertex process $(\eta_t)_{t\geq0}$ on the strip $\{(x+t,t):-(N-k)\leq x\leq k\}$ with $k$ particles, started from the initial condition $\bm{1}_{0<x\leq k}$.
    Then for any time $t$ the law of $\eta_t$ is dominated by $v_{(N-k)^k}$, where $(N-k)^k$ is the rectangular partition with $k$ rows of length $N-k$.
\end{proposition}
\begin{proof}
    We proceed as above, noting that in this case the only $P_j$ appearing in the product are $P_{-N+k+1},\dots, P_{k-1}$.
    All diagrams that can be obtained by flipping boxes in these diagonals are contained in the rectangular diagram $(N-k)^k$.
\end{proof}

\begin{remark}
    Note that for the process on the half line, $\pi^k$ is the stationary measure, and therefore this is another case in which the second half of the censoring theorem from \cite{MR3106501} holds.
    However, for the strip the dominating measure is not the stationary measure.
\end{remark}

\subsection{Labeling processes}
One of our main applications of Theorem \ref{thm:censoring} is to a censoring scheme arising from a multi-class (or colored) stochastic six-vertex process where we censor vertices such that the resulting graph describes the behavior of second-class particles with respect to third-class particles, i.e., where all positions containing first-class particles and holes have been deleted.

\begin{figure}[t]
		\centering
		\begin{tabular}{|c|c|c|c|c|c|}
			\hline
			\begin{tikzpicture}[scale = 1.5]
			\draw[fill][white] (0.5, 0) circle (0.05);
			\draw[thick][white] (0, 0) -- (1,0);
			\draw[thick][white] (0.5, -0.5) -- 
			(0.5,0.5);
			\node at (0.5, 0) {Configuration};
			\end{tikzpicture}
			&
			\begin{tikzpicture}[scale = 1.2]
			\draw[thick, red] (0, 0) -- (1,0);
			\draw[thick, red] (0.5, -0.5) -- (0.5,0.5);
            \draw[fill] (0.5, 0) circle (0.05);
            \node at (-0.2, 0) {$i$};
            \node at (0.5, 0.7) {$i$};
            \node at (1.2, 0) {$i$};
            \node at (0.5, -0.7) {$i$};

			\end{tikzpicture}
			&
			\begin{tikzpicture}[scale = 1.2]
			\draw[thick, blue](0, 0) -- (1,0);
			\draw[thick, red] (0.5, -0.5) -- (0.5,0.5);
			\draw[fill] (0.5, 0) circle (0.05);
            \node at (-0.2, 0) {$j$};
            \node at (0.5, 0.7) {$i$};
            \node at (1.2, 0) {$j$};
            \node at (0.5, -0.7) {$i$};
			\end{tikzpicture}
			&
			\begin{tikzpicture}[scale = 1.2]
			\draw[thick][blue] (0, 0) -- (0.5,0);
			\draw[thick][blue] (0.5, 0) -- (0.5, 0.5);
			\draw[thick, red] (0.5, 0) -- (1, 0);
			\draw[thick, red] (0.5, -0.5) -- (0.5, 0);
			(0.5,0.5);
			\draw[fill] (0.5, 0) circle (0.05);
            \node at (-0.2, 0) {$j$};
            \node at (0.5, 0.7) {$j$};
            \node at (1.2, 0) {$i$};
            \node at (0.5, -0.7) {$i$};
			\end{tikzpicture}
			&
			\begin{tikzpicture}[scale = 1.2]
			\draw[thick, red] (0, 0) -- (1,0);
			\draw[thick][blue] (0.5, -0.5) -- (0.5,0.5);
			\draw[fill] (0.5, 0) circle (0.05);
            \node at (-0.2, 0) {$i$};
            \node at (0.5, 0.7) {$j$};
            \node at (1.2, 0) {$i$};
            \node at (0.5, -0.7) {$j$};
			\end{tikzpicture}
			&
			\begin{tikzpicture}[scale = 1.2]
			\draw[thick, red] (0, 0) -- (0.5,0);
			\draw[thick, red] (0.5, 0) -- (0.5, 0.5);
			\draw[thick, blue] (0.5, 0) -- (1, 0);
			\draw[thick, blue] (0.5, -0.5) -- (0.5, 0);
			(0.5,0.5);
			\draw[fill] (0.5, 0) circle (0.05);
            \node at (-0.2, 0) {$i$};
            \node at (0.5, 0.7) {$i$};
            \node at (1.2, 0) {$j$};
            \node at (0.5, -0.7) {$j$};
			\end{tikzpicture}
			\\
			\hline
			Weight 
			& 1  & $b_1$ & $1- b_1$ & $b_2$ & $1-b_2$\\
			\hline
		\end{tabular}
		\caption{The allowed configurations for the multi-class stochastic six-vertex model, where red lines represent class $i$ and blue lines represent class $j$ for $i < j$. }
		\label{fig:multiClasss6v}
	\end{figure}
    
Define the multi-class stochastic six-vertex process as in \cite[Section 2.3]{MR4093866}.
That is, instead of edges being occupied/unoccupied by paths, we assign every edge a label in $\Z\cup\{\pm\infty\}$, and consider configurations where at every vertex the outgoing edges must have the same labels as the incoming edges.
Vertices are then given weights according to the relative order of the labels as in Figure \ref{fig:multiClasss6v}.
The associated process $(\eta_t)_{t\geq0}$, where $\eta_t:\mathbb Z\to\mathbb Z\cup\{\pm\infty\}$, is called the multi-class stochastic six-vertex process.
To deal with initial conditions which are not constant sufficiently far to the left, one considers so-called cut-vertices, see \cite[Lemma 2.3]{MR4093866} or \cite[Proposition 2.2]{MR5002251}.
By considering initial conditions with only two values of labels, one recovers the single-class process.
We can now state the following theorem.
\begin{theorem}\label{thm:labeling}
    Let $(\eta_t)_{t\geq 0}$ be a multi-class stochastic six-vertex process on the line with parameters $0<b_1<b_2<1$ and with the following initial conditions:
    \begin{itemize}
        \item There are some first-class particles (finitely or infinitely many).
        \item There are infinitely many second-class particles, all to the right of $0$.
        \item There are infinitely many third-class particles, all at or to the left of $0$.
        \item All other positions are empty, i.e., have class $\infty$.
    \end{itemize}
    Let $\cdots < {Z}_t(-1) <  {Z}_t(0)<{Z}_t(1)<\cdots$ be the ordered positions of the second- and third-class particles at time $t$, where $Z_t(0)$ is chosen such that the number of third-class particles among $(Z_t(k))_{k>0}$ equals the number of second-class particles among $(Z_t(k))_{k\leq0}$.
    Let $\xi_t$ denote the configuration obtained by considering the labeling process of the second and third-class particles at time $t$ where a second-class particle receives label $1$ and a third-class particle receives label $0$. In other words, let 
    \begin{equation}
    \xi_t(k) = \begin{cases} 1 & \text{if ${Z}_t(k)$ contains a second-class particle} \\ 0 & \text{if ${Z}_t(k)$ contains a third-class particle, } \end{cases}
    \end{equation}
    for $k\in\mathbb Z$.
    Then for any $t$, the law of $\xi_t$ conditioned on both $\bm{Z} = (Z_s(k))_{ s \leq t,k \in \mathbb Z}$ and the space-time history of the first-class particles, is stochastically dominated by the blocking measure $\pi$. 
\end{theorem}
Note that the choice of $Z_s(0)$ among the second- and third-class particles is the same for all possible realizations of $\xi$, given the paths of the second- and third-class particles.

\begin{remark}
    An analogous result is already known in the setting of ASEP where the proof follows from the monotonicity of the graphical construction of ASEP. This was first stated in \cite[Lemma 4.3]{MR2135733} and proven in detail in \cite{MR2630064}.
\end{remark}

\begin{proof}[Proof of Theorem~\ref{thm:labeling}]
First note that the law of $\xi_0$ is given by $\delta_{\bm{1}_{x>0}}$. Once we condition on $\bm Z$, the labeling process can be understood as the stochastic six-vertex process on the subgraph of $\mathbb Z^2$ given by the edges of the second- and third-class particles.
Indeed, whether particles turn or go straight at a vertex of this subgraph at which a second-class and a third-class particle meet is independent of the rest of $\bm Z$, since second and third-class particles both interact the same way with the first-class particles and the holes. For a more formal description of this independence, see the proof of Theorem 1.6 in \cite{MR5002251}.

The graph can be described as a sequence of directed paths, which touch at the vertices of degree four.
Whenever the paths touch the path on the left will be the horizontal incoming edge and the path on the right will be the vertical incoming edge of the vertex at which they touch.
What matters for the law of the labeling process is only the sequence of degree four vertices where the paths touch.
Any such sequence can be replicated by censoring sufficiently many vertices.
\end{proof}

We can also consider the same setting as Theorem~\ref{thm:labeling}, except with only finitely many second-class particles, or only finitely many third-class particles, which will correspond to the process on the half line in the previous subsection.

\begin{corollary}\label{cor:finitesecondclass}
    Consider the same setup as in Theorem~\ref{thm:labeling}, except there are only $N$ second-class particles.
    Let $\dots<Z_t(N-1)<Z_t(N)$ be the ordered positions of the second- and third-class particles and consider the labeling process
    \begin{equation}
    \xi_t(k) = 
    \begin{cases} 
    1 & \text{if $k\leq N$ and ${Z}_t(k)$ contains a second-class particle} \\ 
    0 & \text{if $k\leq N$ and ${Z}_t(k)$ contains a third-class particle, } \\
    1 &\text{if $k> N$}.\end{cases}
    \end{equation}
    Then for any time $t$, the law of $\xi_t$ is dominated by $\pi^N$.
\end{corollary}

\begin{remark}\label{rem:kone}
   Specializing this corollary to $N=1$ yields Theorem 1.6 in \cite{MR5002251}, which states that the number of third-class particles that overtake a single second-class particle is stochastically dominated by a geometric random variable with parameter $q$.
\end{remark}

\section{Hecke algebra point of view}\label{sec:Hecke}
While working on this project, we discovered that the measures $v_{\lambda}$ are actually a special case of the \emph{parabolic Kazhdan--Lusztig $R$-polynomials} defined by Deodhar \cite{MR916182} and further studied by Brenti in \cite{MR1972246}. This connection arises when viewing the stochastic six-vertex model as a \emph{random walk on a Hecke algebra}. We review the relevant definitions and discuss this connection in more detail.

\subsection{Hecke algebra setup}
We first review the connection between the colored stochastic six-vertex model and random walks on Hecke algebras, following the notation in \cite{bufetov2020interacting}. We will restrict ourselves to the colored stochastic six-vertex model defined on an $i \times (n-i)$ rectangle. We view the colored stochastic six-vertex model as a permutation-valued particle process, i.e., $(\eta_t)_{t\geq0}$ such that $\eta_t \in S_n$ is a permutation of the $n$ particles with colors $1, 2, \ldots, n.$

Note that $S_n$ is a Coxeter group generated by the set $S$ of adjacent transpositions of the form $(k, k+1)$ for some $1 \leq k \leq n-1$. Each permutation $\sigma \in  S_n$ can be assigned a length $\ell(\sigma):= \text{inv}(\sigma)$ which is given by the number of inversions in the permutation $\sigma$. Our convention is to define the product of two permutations as $\tau \sigma =  \sigma \circ \tau.$

We will make use of the following partial order called the \emph{Bruhat order} on $S_n$.

\begin{definition}[Bruhat ordering on permutations]
    For permutations $u, v \in S_n$, we say that $u \to v$ if there exists a transposition $w$ such that $v = w u$ and $\ell(v) > \ell(u)$. We say that $u \leq v$ if there is a finite sequence of permutations $u_i \in S_n$ such that $u = u_1 \to u_2 \to \ldots \to u_m= v$. 
\end{definition}

Let $0 < q < 1$. We define the \emph{Hecke algebra} $\mathcal H(S_n)$ to be the algebra with linear basis $\{T_{\sigma}\}_{\sigma \in S_n}$ where we can define $T_{\sigma}$ using the following multiplication rule for all $\sigma \in S_n, \tau \in S$. 

\begin{equation}\label{eq:HeckeRel}
    T_\tau T_\sigma=\begin{cases}
        T_{\tau\sigma},&\text{ if }\ell(\tau\sigma)=\ell(\sigma)+1\\
        (1-q^{-1})T_\sigma+q^{-1}
        T_{\tau\sigma},&\text{ if }\ell (\tau\sigma)=\ell(\sigma)-1.
    \end{cases}
\end{equation}
An equivalent way to write this is that the $T_{\sigma}$ satisfy
\begin{align} \label{eq:multRule}
    T_{\sigma} T_{\sigma'}&=  T_{\sigma \sigma'} &&\text{ if } \ell(\sigma \sigma') = \ell(\sigma) + \ell(\sigma')  \notag \\
    (T_{\tau} - 1)(T_{\tau} + q^{-1}) &= 0 &&\text{ if } \tau \in S.
\end{align}

\begin{remark} \label{rmk:tildeT}
   We note here that the choice of the multiplication rule in \eqref{eq:multRule} is somewhat non-standard, and is made in order to better match the probabilistic interpretation of the model. However, if we define 
   $\widetilde{T}_{\sigma}:= q^{\ell(\sigma)} T_{\sigma}$ then we can see that the $\widetilde{T}_{\sigma}$ instead satisfy
   \begin{align} 
    (\widetilde{T}_{\tau} - q)(\widetilde{T}_{\tau} + 1) =0.
\end{align}
This matches the more standard convention found e.g. in \cite{MR2133266}\footnote{Note that in \cite{bufetov2020interacting} a third convention is used: $ (T_{\tau} - 1)(T_{\tau} + q) = 0 $. This matches our convention up to swapping $q$ and $q^{-1}$.}.
\end{remark}

Next, we define the set 
$$\mathcal H_{\text{prob}}(S_n) = \{h \in \mathcal H(S_n): h = \sum_{\sigma \in S_n} \kappa_{\sigma}T_{\sigma}, \sum_{\sigma} \kappa_{\sigma} = 1, \kappa_{\sigma} \geq 0\}.$$
Each of the elements $h \in \mathcal H_{\text{prob}}$ can be thought of as defining a probability distribution on $S_n$. 

Let $e \in S_n$ denote the identity permutation. Then for $\tau = (k, k+1) \in S$, we can define the element $Y_{\tau} \in \mathcal H_{\text{prob}}(S_n)  $ as follows: $$Y_{\tau}:= b_1 T_{\tau} + (1-b_1)T_e.$$ Note that 
\begin{equation}\label{eq:YT}
Y_\tau T_\sigma=\begin{cases}
    b_1T_{\tau\sigma}+(1-b_1)T_\sigma\text{, if }\sigma(k)<\sigma(k+1)\\
    b_2T_{\tau\sigma}+(1-b_2)T_\sigma\text{, if }\sigma(k)>\sigma(k+1)\,,
\end{cases}
\end{equation}
which we can check using \eqref{eq:HeckeRel}. It follows that $\mathcal{H}_{\text{prob}}$ is closed under left multiplication by $Y_{\tau}$.

Finally, to define the stochastic six-vertex model at time $1$, we multiply together these elements in a deterministic fashion as follows. Define
$$W_{i,n}: = Y_{(n-1,n)}\cdots Y_{(i+1,i+2)}Y_{(i,i+1)}.$$
Then the colored stochastic six-vertex model in the rectangle $i \times (n-i)$ is given by the element 
$$W_{1, n-i+1}\cdots W_{i-1,n-1}W_{i,n}.$$ 
We call this viewpoint of defining the stochastic six-vertex model through a product of elements in $\mathcal{H}_{\text{prob}}(S_n)$ a \emph{random walk on the Hecke algebra $\mathcal H(S_n)$.}

We now explain how to go from the colored stochastic six-vertex model defined above back to the uncolored stochastic six-vertex model via projection. Given $\sigma\in S_n$, we can obtain a particle-hole configuration by taking $(\mathbf{1}_{\sigma(j)> i})_{j\in \{1, \ldots,  n\}}$. In other words, we map all particles with color less than or equal $i$ to holes and all particles with colors greater than $i$ to particles. We can then define the extended configuration
$$ \eta(x) = \begin{cases}
    0,&\text{if }x\leq -i\\
    \mathbf{1}_{\sigma(x+i)> i}, &\text{if } -i+1\leq x\leq n-i \\
    1 & \text{if } x\geq n-i+1\,.
\end{cases}$$
The particle configuration $\eta(x)$ belongs to the set $A$ defined in \eqref{eq:ADef} and is therefore associated to a partition $\lambda \in \mathcal P(n,n-i)$. We define the map $\phi:S_n\to\mathcal P(n, n-i)$ by setting $\phi(\sigma)=\lambda$. 

By abuse of notation, we also set $\phi(T_\sigma)=\delta_{\phi(\sigma)}$ and extend linearly to a function $\phi:\mathcal{H}_{\text{prob}}(S_n) \to \mathcal{M}_1(\mathcal P(n, n-i))$ so that we can also project probability measures on permutations to probability measures on particle configurations/partitions. Note that if $\sigma \leq \sigma'$ under the Bruhat ordering, then $\phi(\sigma) \leq \phi(\sigma')$ under the containment ordering in \eqref{eq:ordering}.

In the language of Coxeter groups, $\phi$ is the projection of $S_n$ onto a parabolic quotient. Let $\tau_i = (i, i+1)$ and let $J = S \setminus \{\tau_i\}.$ We can then define the \emph{parabolic subgroup} $\left(S_n\right)_J$ as the subgroup of $S_n$ generated by $J$. We can see that 
$\left(S_n\right)_J=\operatorname{Stab}(\{1, \ldots i\}).$  We also define the \emph{parabolic quotient}
\begin{align*}
 \quad\left(S_n\right)^J:&= \left\{\sigma \in S_n:  \ell(\sigma s) >  \ell(\sigma) \text{ for all } s \in J \right\} \\
 &=\left\{\sigma \in S_n: \sigma^{-1}(1)<\cdots<\sigma^{-1}(i) \text { and } \sigma^{-1}(i+1)<\cdots<\sigma^{-1}(n)\right\}.   
\end{align*}
We call this a quotient because the minimal-length representatives of the left cosets of $(S_n)_J$ are given by $(S_n)^J$. In particular the representative of the coset $\sigma \cdot (S_n)_J$ is obtained by placing the colors $1, \ldots, i$ in the positions $\sigma^{-1}(1), \ldots, \sigma^{-1}(i)$ and the colors $i+1, \ldots n$ in positions $\sigma^{-1}(i+1), \ldots, \sigma^{-1}(n)$. Note that $\sigma \in \left(S_n\right)^J$ is completely determined by specifying $\sigma^{-1}(1), \ldots, \sigma^{-1}(i)$ and therefore can be identified with the particle configuration $(\mathbf{1}_{\sigma(j)> i})_{j\in \{1, \ldots,  n\}}$. It follows that we can identify $\left(S_n\right)^J$ with $\mathcal P(n, n-i)$.

We will now define a second basis for the Hecke algebra that projects down to the basis $v_{\lambda}$ defined in Definition \ref{def:vn}. This new basis consists of the elements $\overline{T}_{\sigma}:=T_{\sigma^{-1}}^{-1}$, which we first need to check are well-defined.

For $\tau \in S$, we can check directly that
\begin{align}
T_{\tau}^{-1} = qT_\tau +(1-q)T_e.
\end{align}
It then follows from the fact that any $T_{\sigma}$ can be written as a product of factors $T_{s_j}$ where $s_j \in S$ that we can also invert $T_{\sigma}$ for all $\sigma \in S_n$.

The following proposition demonstrates that  $Y_{\tau}\overline{T}_{\sigma}$ can be written as a convex combination of $\overline{T}_{\tau\sigma}$ and $\overline{T}_{\sigma}$. This property is the analogue of Proposition \ref{prop:vdynamics} and will show that $\phi(\overline{T}_{\sigma}) = v_{\phi(\sigma)}$.
\begin{proposition}\label{prop:YTbar}
Let $\tau=(k, k+1)$. Then the elements $\overline{T}_{\sigma}:=T_{\sigma^{-1}}^{-1}$ satisfy  
\begin{equation}
Y_\tau \overline{T}_\sigma=\begin{cases}\label{eq:YTbar}
    (1-b_2)\overline{T}_\sigma+b_2\overline{T}_{\tau\sigma}\text{, if }\sigma(k)<\sigma(k+1)\\
    (1-b_1)\overline{T}_\sigma+b_1\overline{T}_{\tau\sigma}\text{, if }\sigma(k)>\sigma(k+1)\,.
\end{cases}
\end{equation}
\end{proposition}
\begin{proof}
If $\sigma(k)<\sigma(k+1)$ then $\ell(\tau\sigma)=\ell(\sigma)+1$.
Then
    \[\overline{T}_{\tau\sigma}=T_\tau^{-1}\overline{T}_\sigma=qT_\tau \overline{T}_\sigma+(1-q)\overline{T}_\sigma.\]
    Rearranging terms gives
    \[
    T_\tau \overline{T}_\sigma=q^{-1} \overline{T}_{\tau\sigma}+(1 - q^{-1})\overline{T}_\sigma
    \]
    and finally
    \[
    Y_\tau \overline{T}_\sigma=(b_1 T_{\tau} + (1-b_1)T_e)\overline{T}_\sigma=b_2\overline{T}_{\tau\sigma}+(1-b_2)\overline{T}_\sigma.
    \]
    The other case proceeds similarly.
\end{proof}

\begin{proposition}\label{prop:Tproj}
    For all $\sigma\in S_n$
    \[
    \phi(\overline{T}_\sigma)=v_{\phi(\sigma)}\,.
    \]
\end{proposition}
\begin{proof}
The key observation is that for any $h\in\mathcal H_\text{prob}(S_n)$ and any transposition $\tau_k=(k,k+1)$ we have that
\[
\phi (Y_{\tau_k}h)=P_k\phi(h)\,.
\]
Considering first $h=T_\sigma$ for some $\sigma\in S_n$, this follows from comparing \eqref{eq:YT} and \eqref{eq:PkonDirac} using the definition of $\phi$.
The statement for general $h$ follows by linearity of $\phi$ on $\mathcal H_\text{prob}$.

Applying $\phi$ to both sides of \eqref{eq:YTbar} we obtain that for all $\sigma\in S_n$ and all $k\in\{1,\ldots,n-1\}$ 
\[
P_k\phi(\overline{T}_\sigma)=\begin{cases}
            (1-b_2)\phi(\overline{T}_\sigma)+b_2\phi(\overline{T}_{\tau_k\sigma}),&\text{ if $\sigma(k)<\sigma(k+1)$}\\
            (1-b_1)\phi(\overline{T}_\sigma)+b_1\phi(\overline{T}_{\tau_k\sigma}),&\text{ if $\sigma(k)>\sigma(k+1)$}\,.
        \end{cases}
\]
Using this we can iteratively prove that $\phi(\overline{T}_\sigma)=v_{\phi(\sigma)}$.
First note that $\phi(\overline{T}_e)=\delta_{()}=v_{()}$.
Now assume that the statement has already been proven for some $\sigma\in S_n$ and we want to show it for $\tau_k\sigma$ such that $\ell(\tau_k\sigma)=\ell(\sigma)+1$, i.e., $\sigma(k)<\sigma(k+1)$.
Then we have that
\[
\phi(\overline{T}_{\tau_k\sigma})=\frac{1}{b_2}P_k\phi(\overline{T}_\sigma)+\frac{(b_2-1)}{b_2}\phi(\overline{T}_\sigma)=\frac{1}{b_2}P_kv_{\phi(\sigma)}+\frac{(b_2-1)}{b_2}v_{\phi(\sigma)}.
\]
Using Proposition~\ref{prop:vdynamics} one can see that the right-hand side of this is equal to $v_{\phi(\tau_k\sigma)}$ since $\phi(\tau_k\sigma)=\phi(\sigma)^k$.
\end{proof}

It follows from Proposition~\ref{prop:Tproj} that $v_{\phi(\sigma)}(\mu)$ is the sum of the change-of-basis coefficients of $T_{\sigma'}$ in $\overline{T}_\sigma$ where the sum runs over $\sigma'$ such that $\phi(\sigma')=\mu$.
The coefficients in the change-of-basis from $\{T_{\sigma}\}_{\sigma \in S_n}$ to $\{\overline{T}_{\sigma}\}_{\sigma \in S_n}$ are given by the \emph{Kazhdan--Lusztig $R$-polynomials} defined by Kazhdan and Lusztig in \cite{MR560412}.
The sums will be given by the \emph{parabolic Kazhdan--Lusztig R-polynomials} defined by Deodhar in \cite{MR916182}.
The following exposition closely follows \cite{MR2133266}.

\begin{definition}[{\cite[Theorem 5.1.1]{MR2133266}}]
The Kazhdan--Lusztig $R$-polynomials are the unique family of polynomials $\left\{R_{u, v}(q)\right\}_{u, v \in S_n} \subseteq \mathbb Z[q]$ defined through the following recurrence relations: Let $R_{u, v}(q) = 0$ if $u \not \leq v$ and $R_{u, v}(q) = 1$ if $u = v$. Finally for $u < v$, and for any $s \in S$ such that $sv < v$,  we have 
\begin{align*}
    R_{u,v}(q)  = \begin{cases}
        R_{su, sv}(q) & \text{if $su < u$} \\
        qR_{su, sv}(q) + (q-1)R_{su, v}(q) & \text{if $su > u$}.
    \end{cases}
\end{align*}
\end{definition}

It is straightforward to check that the following identity holds:
\begin{equation}\label{eq:RIdentity}
(-q)^{\ell(v) - \ell(u)} R_{u, v}\left(\frac{1}{q}\right)=R_{u, v}(q).
\end{equation}
We can now express the elements $\overline{T}_\sigma = T_{\sigma^{-1}}^{-1}$ in terms of the $T_{\sigma}$ as follows:
\begin{proposition}[{\cite[Chapter 6.1]{MR2133266}}] \label{prop:basis} Recall the definition of $\widetilde{T}_{\sigma} = q^{\ell(\sigma)}T_{\sigma}$ from Remark \ref{rmk:tildeT}. Then we have 
\begin{equation*}
\widetilde{T}_{w^{-1}}^{-1}=q^{-\ell(w)} \sum_{y \leq w}(-1)^{\ell(y) - \ell( w)} R_{y, w}(q) \widetilde{T}_y.
\end{equation*}
Equivalently, using \eqref{eq:RIdentity}, we have  
\begin{align*}
{T}_{w^{-1}}^{-1}&=\sum_{y \leq w}(-1)^{\ell(y) - \ell( w)} R_{y, w}(q) q^{\ell(y)} {T}_y\\
&= \sum_{y \leq w}q^{\ell(w)}R_{y, w}(\tfrac1q) {T}_y.
\end{align*}
\end{proposition}

We can use the Kazhdan--Lusztig $R$-polynomials to define permutation-indexed probability measures on $S_n$. These are exactly the ``lifted" versions of the $v_{\lambda}$ to $S_n$. 

\begin{proposition} \label{prop:probmeasure}
   For a fixed $q < 1, v \in S_n$, $\mathbb P_v(u):= q^{\ell(v)}R_{u, v}(\tfrac1q)$ defines a probability measure over permutations $u \in S_n$.
\end{proposition}

\begin{remark}
     This was also pointed out by Galashin in \cite{MR4317703} in a related context.
\end{remark}

\begin{proof}
First note that for $q \geq 1$, we have that $R_{u, v}(q) \geq 0$, see e.g. \cite[Proposition 5.3.1]{MR2133266}. This is only true for $q \geq 1$ and the $R$-polynomials can be negative for $q < 1$. 

Finally, Proposition \ref{prop:basis} tells us that
\begin{align*}
{T}_{w^{-1}}^{-1}&=\sum_{y \leq w}q^{\ell(w)}R_{y, w}(\tfrac1q) {T}_y.
\end{align*}
Consider the algebra homomorphism from the Hecke algebra $\mathcal H(S_n)$ to $\mathbb Z[q, q^{-1}]$ that maps $T_w \mapsto 1$ for all $w$. This is a homomorphism as it respects the quadratic relation  $(T_{\tau} - 1)(T_{\tau} + q^{-1}) = 0$. It follows that \begin{align*}
    \sum_{u} q^{\ell(v)}R_{u, v}(\tfrac 1q) = 1,
\end{align*} 
so that $\mathbb P_v(u)$ is a probability measure.
\end{proof}

We now define the parabolic Kazhdan--Lusztig $R$-polynomials.
\begin{definition} [\cite{MR916182}]
    Given $J \subseteq S$ and $u, v \in (S_n)^J$, we can define the parabolic Kazhdan--Lusztig $R$-polynomials with parameter $x \in \{-1, q\}$ as $$
R_{u, v}^{J, x}(q)=\sum_{w \in (S_n)_J} (-x)^{\ell(w)}R_{uw, v}(q).
$$
\end{definition}

We will consider the parabolic Kazhdan--Lusztig $R$-polynomials with $J = S \setminus \{\tau_i\}$ and with $x = -1$ so that we have
$$
R_{u, v}^{J, -1}(q)=\sum_{w \in (S_n)_J} R_{ uw, v}(q) =\sum_{z \in u \cdot (S_n)_J} R_{z, v}(q) .
$$
In other words, we are summing over all permutations in the coset $u \cdot (S_n)_J$, i.e., all permutations $z$ such that $\phi(z) = \phi(u).$ The following proposition is now immediate:
\begin{proposition} \label{prop:parabolicProjection}
We have 
$$v_{\phi(v)}(\phi(u)) = q^{\ell(v)} R^{J, -1}_{u, v}\left(\frac1q\right)  = q^{\ell(u)}(-1)^{\ell(v)- \ell(u)}R^{J, q}_{u, v}\left(q\right).$$
\end{proposition}

\begin{proof}
    We have
    \begin{align*}
        v_{\phi(v)}(\phi(u))&= \sum_{z:\phi(z)=\phi(u)}\mathbb P_v(z)\\
        &= \sum_{z:\phi(z)=\phi(u)} 
q^{\ell(v)}R_{z, v}(\tfrac1q) \\
&= q^{\ell(v)} R^{J, -1}_{u, v}\left(\frac1q\right).
    \end{align*}
The identity
$$q^{\ell(v)} R^{J, -1}_{u, v}\left(\frac1q\right)  = q^{\ell(u)}(-1)^{\ell(v)- \ell(u)}R^{J, q}_{u, v}\left(q\right)$$ follows from \cite[Proposition 2.5]{MR1972246}.
\end{proof}

\subsection{Lack of monotonicity for the Kazhdan--Lusztig \texorpdfstring{$R$}{R}-polynomials} It is natural to ask whether the monotonicity properties we showed for the \emph{parabolic} Kazhdan--Lusztig $R$-polynomials in Proposition \ref{prop:monotonicity} are just projections of some more general monotonicity properties for the Kazhdan--Lusztig $R$-polynomials. We give a partial negative answer below. 

We show through a counterexample that the measures $\mathbb P_v(u)$ do not exhibit the desired stochastic monotonicity relative to the Bruhat ordering on $v$.

\begin{example}
  Consider the symmetric group $S_3$. We will show that even though $321 > 231$ in the Bruhat ordering, the measure $\mathbb P_{321}$ does not stochastically dominate $\mathbb P_{231}$. We first compute the polynomials $R_{\bullet, 231}$ and $R_{\bullet, 321}$ to obtain

\begin{align*}
    R_{123, 231}(q)&= (q-1)^2; && R_{132, 231}(q)= q-1;\\
    R_{213, 231}(q)&= q-1; &&  R_{231, 231}(q) = 1,
\end{align*}

and

\begin{align*}
    R_{123, 321}(q)&=(q-1)(q^2-q+1); &&R_{132, 321}(q)=(q-1)^2; \\
    R_{213, 321}(q)&=(q-1)^2; &&R_{231, 321}(q)=q-1; \\
    R_{312, 321}(q)&=q-1; &&R_{321, 321}(q)=1.
\end{align*}
It follows that for all $0 < q < 1$,
$$\mathbb P_{321}(123)=  q^3(q^{-1}-1)(q^{-2}-q^{-1}+1) > \mathbb P_{231}(123)= q^2(q^{-1}-1)^2.$$
This implies that $\mathbb P_{321}$ does not stochastically dominate $\mathbb P_{231}$.

\end{example}

\appendix
\section{Full censoring inequality for \texorpdfstring{$b_1+b_2\leq 1$}{b1+b2<=1}}\label{sec:ASEPregime}
In this appendix we will outline how one can prove the full censoring inequality in the case of $b_1+b_2\leq 1$. We first note that for parameters $b_1$ and $b_2$ satisfying $b_1+b_2>1$ the full analogue of Theorem~\ref{thm:PWcensoring} does not hold, since additional updates can move one further away from stationarity, as the following example shows.
\begin{example}\label{ex:censoring}
    Let $b_1+b_2>1$.
    Then, for the censoring schemes
    \[
    \mathcal C_1=\{(-1,0),(1,0),(0,1),(2,1)\},\quad \mathcal C_2=\mathcal C_1\cup\{(1,1)\}\,,
    \]
    the law of $\eta_2$ under the censoring of $\mathcal C_1$ is strictly dominated by the law of $\eta_2$ under the censoring of $\mathcal C_2$, even though $\mathcal C_1\subsetneq\mathcal C_2$.
    Indeed, the first one has probability $(1-b_1)^2+b_1b_2$ to be in minimal state $\bm{1}_{x>0}$, while the second has probability $1-b_1$ to be in this state.
    For $b_1+b_2>1$ it holds that $(1-b_1)< (1-b_1)^2+b_1b_2$.
\end{example}

The following theorem is the censoring inequality stated for the stochastic six-vertex model.
Its proof will closely follow \cite{MR3474475}, except for the proof of the monotonicity property Lemma~\ref{lem:monotonicity}, which proceeds via a coupling particular to the stochastic six-vertex model.



\begin{theorem}\label{thm:censoringASEP}
    Let $b_1\leq b_2$ satisfy $b_1+b_2\leq 1$ and let $(\eta_t)_{t\geq0}$ and $(\tilde\eta_t)_{t\geq0}$ be the $\mathcal C$-censored and the $\tilde{\mathcal C}$-censored stochastic six-vertex processes respectively, where $\mathcal C\subset\tilde{\mathcal C}$.
    Then
    \[
    \mathcal L(\tilde{\eta}_t)\preceq\mathcal L(\eta_t)\preceq \pi
    \]
    and
    \[
    \|\mathcal L(\tilde{\eta}_t)-\pi\|_\textrm{TV}\geq\|\mathcal L(\eta_t)-\pi\|_\textrm{TV}\,.
    \]
\end{theorem}

The proof of this will follow from the following two properties of the $P_k$.
\begin{lemma}\label{lem:decreasing}
    If $b_1+b_2\leq 1$ the following holds.
    Let $\nu$ be a measure on $A$ such that the Radon--Nikodym derivative $\frac{\mathrm{d}\nu}{\mathrm{d}\pi}$ of $\nu$ with respect to $\pi$ is a decreasing function on $A$. Then, for all $k\in\mathbb Z$ the Radon--Nikodym derivative $\frac{\mathrm{d}P_k\nu}{\mathrm{d}\pi}$ is also a decreasing function and $\nu\preceq P_k\nu$.
\end{lemma}
\begin{proof}
    Note that since $\pi$ is proportional to $q^{|\lambda|}$, $\frac{\mathrm{d}\nu}{\mathrm{d}\pi}$ being decreasing means that for any $\lambda$ such that $\lambda^k>\lambda$ we have
    \[
    \nu(\lambda^k)\leq q\nu(\lambda)\,.
    \]
    After applying $P_k$ we have    
    \begin{align*}
        P_k\nu(\lambda^k)-qP_k\nu(\lambda)&=b_1\nu(\lambda)+(1-b_2)\nu(\lambda^k)-q(b_2\nu(\lambda^k)+(1-b_1)\nu(\lambda))\\
        &=(1-b_1-b_2)(\nu(\lambda^k)-q\nu(\lambda))\leq 0\,.
    \end{align*}
    Notice that the same calculation shows that $P_k\nu(\lambda^k)\geq \nu(\lambda^k)$, since $P_k\nu(\lambda)+P_k\nu(\lambda^k)=\nu(\lambda)+\nu(\lambda^k)$.
    Note that $P_k\nu$ and $\nu$ only differ at partitions $\lambda$ such that $\lambda^k\neq \lambda$ and at these points $P_k\nu$ always assigns more mass to the bigger partition, so we also get $\nu\preceq P_k\nu$.
    
\end{proof}

\begin{lemma}\label{lem:monotonicity}
    If $b_1+b_2\leq 1$ and $\nu_1\preceq\nu_2$ then $P_k\nu_1\preceq P_k\nu_2$.
\end{lemma}
\begin{proof}
    Let $\chi$ be a coupling of $\nu_1$ and $\nu_2$ such that for $\eta^1_0,\eta^2_0$ sampled according to $\chi$ we have that $\eta^1_0\leq\eta^2_0$.
    Let $\eta^1_1$ and $\eta^2_1$ be obtained from $\eta^1_0$ and $\eta^2_0$ respectively by applying the dynamics described in Definition~\ref{def:Pk}, i.e., the values $(\eta^1_0(k),\eta^1_0(k+1))$ are swapped with probability $b_1$ if they are $(0,1)$ and with probability $b_2$ if they are $(1,0)$.
    We will now construct a coupling of $\eta^1_1$ and $\eta^2_1$ such that $\eta^1_1\leq\eta^2_1$ holds, which shows that $P_k\nu_1\preceq P_k\nu_2$.
    Firstly, if $(\eta^1_0(k),\eta^1_0(k+1))=(\eta^2_0(k),\eta^2_0(k+1))$, sample $\eta^1_1$ and $\eta^2_1$ such that $(\eta^1_1(k),\eta^1_1(k+1))=(\eta^2_1(k),\eta^2_1(k+1))$ as well.
    Further if $(\eta^1_0(k),\eta^1_0(k+1),\eta^2_0(k),\eta^2_0(k+1))=(0,1,1,0)$ sample $(\eta^1_1(k),\eta^1_1(k+1),\eta^2_1(k),\eta^2_1(k+1))$ according to
    \[
    \mathbb P\big((\eta^1_1(k),\eta^1_1(k+1),\eta^2_1(k),\eta^2_1(k+1))=x\big)=\begin{cases}
        b_1\text{, if }x=(1,0,1,0)\\
        b_2\text{, if }x=(0,1,0,1)\\
        1-b_1-b_2\text{, if }x=(0,1,1,0)\\
        0\text{, if }x=(1,0,0,1)\,.
    \end{cases}
    \]
    One can easily check that this has the correct marginals.
    Importantly the event that a particle moves to the left in $\eta^1$, while a particle moves to the right in $\eta^2$ at the same time is $0$, which would be the only case in which the ordering between $\eta^1$ and $\eta^2$ could be broken.
\end{proof}

Given these two lemmas Theorem~\ref{thm:censoringASEP} now follows as in \cite[Section A.2]{MR3474475}.
We repeat the argument here for completeness.

\begin{proof}[Proof of Theorem~\ref{thm:censoringASEP}]
    By Proposition~\ref{prop:inter}, what we need to show is that for all $N$
    \begin{equation}\label{eq:productinequality}
        \prod_{n=1}^N \hat{P}_{k_n}\delta_{\bm{1}_{x>0}}\preceq\prod_{n=1}^N P_{k_n}\delta_{\bm{1}_{x>0}}\,,
    \end{equation}
    where $$\hat{P}_{k_n}=\begin{cases}\mathrm{Id}&\text{for $k_n$ corresponding to vertices censored in $\tilde{\mathcal C}$ but not in $\mathcal C$.}\\
    P_{k_n} &\text{else}.\end{cases}$$

    By Lemma~\ref{lem:decreasing} all partial products are measures with decreasing Radon--Nikodym derivatives with respect to $\pi$, since $\delta_{\bm{1}_{x>0}}$ is such a measure.
    In particular, this shows they are all dominated by $\pi$.
    
    We will now prove \eqref{eq:productinequality} by induction on $N$.
    Assume it has been proven for some $N$, and denote the product until $N$ on the left and right hand side by $\tilde\nu_N$ and $\nu_N$ respectively.
    If $\hat P_{k_{N+1}}$ is the identity we have
    \[
    \tilde\nu_{N+1}=\tilde\nu_N\preceq\nu_N\preceq P_{k_{N+1}}\nu_N=\nu_{N+1}\,,
    \]
    where we used the induction hypothesis and Lemma~\ref{lem:decreasing}.
    If on the other hand $\hat P_{k_{N+1}}=P_{k_{N+1}}$ then
    $\tilde\nu_{N+1}\preceq\nu_{N+1}$ follows from $\tilde\nu_N\preceq\nu_N$ and Lemma~\ref{lem:monotonicity}.

    To obtain the statement on the total variation distance note that for decreasing measures $\nu_1$ and $\nu_2$, the inequality $\nu_1\preceq\nu_2$ implies that $\|\nu_1-\pi\|_\textrm{TV}\geq\|\nu_2-\pi\|_\textrm{TV}$, see \cite[Lemma A.4]{MR3474475}.
\end{proof}

To compare with the proof for general $b_1<b_2$, note particularly that the parabolic Kazhdan--Lusztig R-measures $v_\lambda$ do not generally have decreasing Radon--Nikodym derivative with respect to $\pi$. 

\begin{remark}
    With the appropriate (very minor) modifications, the same proof works for the permutation-valued stochastic six-vertex model on a strip.
\end{remark}

\bibliographystyle{alpha}
\bibliography{refs.bib}

\end{document}